\documentclass{amsart}

\usepackage{amsmath,amssymb,amscd}
\usepackage{hyperref}

\usepackage{comment}
\usepackage{graphicx,color}
\usepackage[all]{xy}

\usepackage[initials,alphabetic]{amsrefs}

\newtheorem{theorem}{Theorem}[section]
\newtheorem{lemma}[theorem]{Lemma}

\theoremstyle{definition}
\newtheorem{definition}[theorem]{Definition}
\newtheorem{remark}[theorem]{Remark}
\newtheorem{notation}[theorem]{Notation}
\newtheorem{example}[theorem]{Example}

\newcommand\cB{\mathcal{B}}
\newcommand\cD{\mathcal{D}}
\newcommand\cH{\mathcal{H}}
\newcommand\cM{\mathcal{M}}
\newcommand\cO{\mathcal{O}}

\newcommand\cX{\mathcal{X}}

\renewcommand\AA{\mathbb{A}}
\newcommand\CC{\mathbb{C}}
\newcommand\NN{\mathbb{N}}
\newcommand\PP{\mathbb{P}}
\newcommand\QQ{\mathbb{Q}}
\newcommand\RR{\mathbb{R}}
\newcommand\ZZ{\mathbb{Z}}

\newcommand\rH{\mathrm{H}}
\newcommand\rT{\mathrm{T}}

\newcommand{\frakm}{\mathfrak{m}}
\newcommand{\frakq}{\mathfrak{q}}

\newcommand{\phiv}{\varphi}

\newcommand{\inv}{^{-1}}
\newcommand{\into}{\hookrightarrow} 
\newcommand{\onto}{\twoheadrightarrow} 

\DeclareMathOperator{\rank}{rank} 
\DeclareMathOperator{\Hom}{Hom} 
\DeclareMathOperator{\Ext}{Ext}
\DeclareMathOperator{\Spec}{Spec} 
\DeclareMathOperator{\Proj}{Proj} 
\DeclareMathOperator{\QCoh}{QCoh} 

\def\conv#1{\mathrm{conv} \left\langle #1  \right\rangle} 

\newcommand{\vertices}{\mathrm{vert}} 
\def\cone#1{\mathrm{cone} \left\langle #1 \right\rangle}
\newcommand{\rec}{\mathrm{rec}} 
\newcommand{\mut}{\mathrm{mut}} 
\def\toric#1#2{\mathrm{TV}_{#1}(#2)} 
\newcommand{\interior}{\mathrm{int}} 

\title{Homogeneous deformations of toric pairs}

\author{Andrea Petracci}

\address{Institut f\"ur Mathematik, Freie Universit\"at Berlin, Arnimallee 3, 14195 Berlin, Germany}

\email{andrea.petracci@fu-berlin.de}

\begin{document}

\begin{abstract}
We extend the Altmann--Mavlyutov construction of homogeneous deformations of affine toric varieties to the case of toric pairs $(X, \partial X)$, where $X$ is an affine or projective toric variety and $\partial X$ is its toric boundary. As an application, we generalise a result due to Ilten to the case of Fano toric pairs.
\end{abstract}

\maketitle

\section{Introduction}

An important trend in modern algebraic geometry is to study pairs consisting of a variety with a divisor. Recent work by Gross--Hacking--Keel~\cite{ghk} suggests that Mirror Symmetry, which was originally formulated for Calabi--Yau varieties, is better understood as a correspondence between \emph{log Calabi--Yau pairs}, i.e.\ pairs $(X, B)$ where $X$ is a variety and $B$ is an effective divisor such that $K_X + B$ is linearly trivial. \emph{Toric pairs} -- that is, pairs $(X,\partial X)$ where $X$ is a toric variety with toric boundary $\partial X$ -- are one of the simplest examples of log Calabi--Yau pairs, and can be understood to lie at the boundary of the moduli space of log Calabi--Yau pairs. It is therefore interesting to understand deformations of toric pairs in this setting.

The aim of this paper is to construct deformations of toric pairs via combinatorial methods of toric geometry, by generalising the constructions due to Altmann \cite{altmann_minkowski_sums, altmann_one_parameter_families} and Mavlyutov \cite{mavlyutov}. The deformations we construct are \emph{homogeneous} with respect to the action of the torus (see Remark~\ref{rmk:homogeneous}) and unobstructed.

After surveying the work of Mavlyutov \cite{mavlyutov} on deformations of affine toric varieties, we extend his construction to deformations of affine toric pairs (Theorem~\ref{thm:deformations_affine_toric_pairs_easy}).
More precisely, if $X$ is an affine toric variety without torus factors and $\partial X$ is its toric boundary, then from some combinatorial input (which we call $\partial$-deformation datum) we construct a formal deformation of the closed embedding $\partial X \into X$ over a power series ring in finitely many variables over $\CC$.  The construction of the deformation is achieved by constructing an affine toric variety $\tilde{X}$ and a closed embedding $X \into \tilde{X}$ and by deforming the equations of this closed embedding.

By applying Proj to this construction, we also construct deformations of projective toric pairs (Theorem~\ref{thm:deformations_projective_toric_pairs_easy}). Both in the affine and in the projective case, the deformations we construct lie inside a bigger toric variety $\tilde{X}$ and are explicit in terms of Cox coordinates of $\tilde{X}$; therefore, in specific examples, it is easy to check if we get smoothings.

Finally, we apply our construction of deformations of projective toric pairs to a particular case which arises in the study of Mirror Symmetry for Fano varieties \cite{mirror_symmetry_and_fano_manifolds}: in this way we are able to reprove and extend an important result of Ilten \cite{ilten_sigma} about families of Fano varieties coming from a combinatorial procedure on Fano polytopes called ``mutation'' (Theorem~\ref{thm:mutations_induce_deformations_easy}).

Now we give a more detailed account of what we do.

\subsection{Minkowski decompositions and deformations of affine toric pairs}

Let $\sigma$ be a full dimensional strongly convex rational polyhedral cone inside a lattice $N$ and let $X = \toric{\CC}{\sigma}$ be the affine toric variety over $\CC$ associated to $\sigma$. Klaus Altmann has extensively studied the deformation theory of $X$. In \cite{altmann_computation_tangent} he computes the tangent space $\rT^1_X$ to deformations of $X$. In \cite{altmann_versal_deformation} he describes the miniversal deformation of $X$ when it is an isolated Gorenstein singularity. In \cite{altmann_minkowski_sums} he notices that Minkowski decompositions of a polyhedron inside $\sigma$, under some hypotheses, induce certain  deformations of $X$;
for example, the two Minkowski decompositions of the standard hexagon (Figure~\ref{fig:esagono}) \begin{figure}
\centering
\def\svgwidth{10cm}
\input{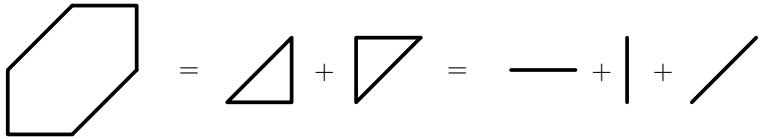}
\caption{The two Minkowski decompositions of the hexagon \label{fig:esagono}}
\end{figure} induce two different deformations of the anticanonical affine cone over the smooth del Pezzo surface of degree $6$. In \cite{altmann_one_parameter_families} he constructs deformations of $X$ from Minkowski decompositions of more general polyhedra inside the cone $\sigma$.

In \cite{mavlyutov} Anvar Mavlyutov gives a unified description of all Altmann's deformations thanks to the use of Cox coordinates. His construction has the same strategy as Altmann's: starting from a Minkowski decomposition of some polyhedron (with some assumptions) one embeds the considered affine toric variety into a larger affine toric variety and then deforms the equations of this closed embedding. More specifically, starting from a Minkowski decomposition of a polyhedron inside the cone $\sigma$ one can construct a bigger cone $\tilde{\sigma}$ in a bigger lattice $\tilde{N}$ and embed the toric variety $X$ associated to $\sigma$ inside the toric variety $\tilde{X} = \toric{\CC}{\tilde{\sigma}}$ associated to $\tilde{\sigma}$ via binomial equations in the Cox coordinates of $\tilde{X}$; by deforming these binomial equations with extra monomials one may produce a deformation of $X$. The precise statement is a theorem of Mavlyutov \cite{mavlyutov} and is rewritten in \S\ref{sec:deformations_affine_mavlyutov} together with a detailed proof. There the Minkowski decomposition is encoded in the notion of a deformation datum (see Definition~\ref{def:deformation_datum}).

We have noticed that Mavlyutov's construction can be applied also to deform the affine toric pair $(X, \partial X)$, where $\partial X$ is the toric boundary of $X$. More precisely, we construct a formal deformation of the closed embedding $\partial X \into X$ over a power series ring $A$ in finitely many variables over $\CC$. By a \emph{formal deformation} of $\partial X \into X$ over $A$ we mean a collection made up of a commutative diagram
\[
\xymatrix{
B_n \ar@{^(->}[r] \ar[rd] & X_n \ar[d] \\
& \Spec A/\frakm_A^{n+1}
}
\]
for each $n \in \NN$, where $\frakm_A$ is the maximal ideal of $A$, $B_n \into X_n$ is a closed embedding, $B_n$ and $X_n$ are flat over $A / \frakm_A^{n+1}$, and all these diagrams are required to be compatible in the following way: the $0$th diagram is just the embedding $\partial X \into X$ over $\Spec \CC$, and the base change of the $(n+1)$th diagram along the closed embedding induced by $A/\frakm_A^{n+2} \onto A/\frakm_A^{n+1}$ is isomorphic to the $n$th diagram.

Our main theorem, which significantly rests on \cite{mavlyutov}, is the following.

\begin{theorem} \label{thm:deformations_affine_toric_pairs_easy}
Let $X$ be an affine toric variety without torus factors and let $\partial X$ be its toric boundary.
Given a Minkowski decomposition of a polyhedron satisfying certain conditions, one can construct a formal deformation of the pair $(X, \partial X)$ over a power series ring in finitely many variables over $\CC$. (See Theorem~\ref{thm:deformations_affine_toric_pairs} for the precise statement.)
\end{theorem}

Example~\ref{ex:cA1} shows how to use this theorem to deform the 3-fold toric $cA_1$ singularity $\Spec \CC[x,y,z,u]/(xy-u^2)$ together with its toric boundary.

\subsection{Deformations of projective toric pairs}

The deformation theory of complete toric varieties is not fully understood yet. When $X$ is a smooth complete  toric variety, Nathan Ilten \cite{ilten_deformations_smooth_toric_surfaces} has computed the tangent space $\rT^1_X$ to deformations of $X$. But, when $X$ is a singular complete toric variety, the tangent space $\rT^1_X$ and the miniversal deformation of $X$ are unknown in general.

Nonetheless, in the literature there are some constructions of homogeneous deformations of toric varieties. 
 Ilten and Vollmert \cite{deformations_T_varieties} construct deformations of rational $T$-varieties of complexity $1$, which are a generalisation of toric varieties \cite{geometry_T_varieties, altmann_hausen_affine_T_varieties, altmann_hausen_suess_T_varieties}. Hochenegger and Ilten \cite{hochenegger_ilten} construct deformations of a rational complexity-1 T-variety together with a T-line bundle. Mavlyutov \cite{mavlyutov} uses Minkowski decompositions of polyhedral complexes in order to construct homogeneous  deformations. Laface and Melo \cite{laface_melo} construct deformations of smooth complete toric varieties by using their Cox rings.

Here, by avoiding the languages of T-varieties and of Cox rings (or more precisely by sweeping them under the carpet), we propose an explicit construction of deformations of polarised projective toric varieties together with their toric boundaries. These deformations live in an ambient projective toric variety $\tilde{X}$ and are completely explicit in terms of the Cox coordinates of $\tilde{X}$.

Our strategy consists in deforming a projective toric variety $X$ by deforming its affine cone $C$ with respect to an ample torus-invariant $\QQ$-Cartier $\QQ$-divisor $D$ on $X$. Deforming a polarised projective variety by deforming its affine cone have already appeared in the literature, e.g.\ \cite{pinkham_deformations_C*action, piene_schlessinger}; in our toric context we took the idea from a specific case in \cite{ilten_sigma}.  More specifically, if the fan of $X$ is in the lattice $N$ of rank $n$, then the section ring
\[
\bigoplus_{k \in \NN} \rH^0(X, \cO_X(\lfloor kD \rfloor))
\]
coincides with $\CC[\tau^\vee \cap M_0]$, where $\tau$ is an $(n+1)$-dimensional strongly convex rational polyhedral cone in the lattice $N_0 = N \oplus \ZZ e_0$ such that $e_0$ is in the interior of $\tau$. We refer the reader to \S\ref{sec:polarised_projective_toric_varieties} for more details about the relationship between the pair $(X,D)$ and the cone $\tau$. Starting from a Minkowski decomposition of a polyhedron inside $\tau$ and by applying Mavlyutov's construction (\cite{mavlyutov} and Theorem~\ref{thm:mavlyutov}), we can deform $C = \Spec \CC[\tau^\vee \cap M_0]$, which is the affine cone over $X$; by applying Proj we construct a deformation of $X = \Proj \CC[\tau^\vee \cap M_0]$. 
Theorem~\ref{thm:deformations_projective_toric_varieties} expresses this deformation via explicit equations in Cox coordinates of a projective variety $\tilde{X}$.

If, in addition, the divisor $D$ is a $\ZZ$-divisor, then we are able to deform also the toric boundary of $X$. This is the content of the following theorem.

\begin{theorem} \label{thm:deformations_projective_toric_pairs_easy}
Let $X$ be a projective toric variety with toric boundary $\partial X$. Given an ample torus-invariant $\QQ$-Cartier $\ZZ$-divisor on $X$ and a Minkowski decomposition of a polyhedron satisfying certain conditions, one can construct a deformation of the pair $(X, \partial X)$ over a power series ring in finitely many variables over $\CC$. (See Theorem~\ref{thm:deformations_projective_toric_pairs} for the precise statement.)
\end{theorem}

\subsection{Mutations and deformations of Fano toric pairs} \label{sec:intro_mutations_deformations}

A \emph{Fano polytope} in a lattice $N$ is a full dimensional polytope $P \subseteq N_\RR$ such that the origin $0 \in N$ lies in the strict interior of $P$ and every vertex of $P$ is a primitive lattice point.
The \emph{spanning fan} (also called the \emph{face fan}) of a Fano polytope $P \subseteq N_\RR$, i.e.\ the fan in $N_\RR$ whose cones are the cones over the faces of $P$, determines a toric variety $X_P$ which is Fano, i.e.\ its anticanonical divisor is $\QQ$-Cartier and ample. This establishes a bijective correspondence between Fano $T_N$-toric varieties and Fano polytopes in $N$. 

Starting from a primitive vector $w \in M$ and a polytope $F \subseteq w^\perp \subseteq N_\RR$ satisfying certain conditions (see Definition~\ref{def:factor}) with respect to the Fano polytope $P \subseteq N_\RR$, it is possible to construct another Fano polytope $P' := \mathrm{mut}_{w,F}(P) \subseteq N_\RR$ (see Definition~\ref{def:mutation}). This procedure is called \emph{mutation} \cite{sigma} and its motivation lies in the study of Mirror Symmetry for Fano varieties \cite{mirror_symmetry_and_fano_manifolds, conjectures}.

It was observed by Nathan Ilten \cite{ilten_sigma} that if two Fano polytopes $P$ and $P'$ in $N_\RR$ are related by a mutation then the corresponding Fano varieties $X_P$ and $X_{P'}$ are two closed fibres of a flat family over $\PP^1$. Ilten's construction relies on the theory of deformations of T-varieties developed in \cite{deformations_T_varieties}.

Here we apply our Theorem~\ref{thm:deformations_projective_toric_pairs_easy} (i.e.\ Theorem~\ref{thm:deformations_projective_toric_pairs}) to this case because the toric boundary $\partial X_P$ is an ample $\QQ$-Cartier $\ZZ$-divisor on $X_P$ and the combinatorial conditions in the definition of mutation allows us to construct a $\partial$-deformation datum. We will show that $X_P$ and $X_{P'}$ are two fibres of the flat family of divisors defined by a trinomial in the Cox coordinates of a projective toric variety of dimension $\dim X_P + 1$. In addition to what was done by Ilten, we can show that also the toric boundary $\partial X_P$ deforms to $\partial X_{P'}$.

\begin{theorem} \label{thm:mutations_induce_deformations_easy}
Let $P$ and $P'$ be two Fano polytopes related by a mutation.
Let $X_P$ (resp.\ $X_{P'}$) be the Fano toric variety associated to the spanning fan of $P$ (resp.\ $P'$) and let $\partial X_P$ (resp.\ $\partial X_{P'}$) be the toric boundary of $X_P$ (resp.\ $X_{P'}$). Then there exists a commutative diagram
\[
\xymatrix{
\cB \ \ar[rd] \ar@{^{(}->}[r] & \cX \ar[d] \\
& V
}
\]
such that $V$ is an open subscheme of $\PP^2_\CC$, the morphism $\cB \into \cX$ is a closed embedding, the morphisms $\cB \to V$ and $\cX \to V$ are projective and flat, and there are two closed points in $V$ for which the base change of the diagram to them are the closed embeddings $\partial X_P \into X_P$ and $\partial X_{P'} \into X_{P'}$ over $\Spec \CC$, respectively.
\end{theorem}

Very roughly speaking, Ilten's result says that mutations of Fano polytopes create a $1$-dimensional skeleton in the moduli space of Fano varieties. Our theorem extends this interpretation to moduli of log Calabi--Yau pairs $(X,B)$ where $X$ is Fano.

The precise constructions of $V$, $\cB$ and $\cX$ in Theorem~\ref{thm:mutations_induce_deformations_easy} are given in Theorem~\ref{thm:mutations_induce_deformations}. We refer the reader to Example~\ref{ex:P2_P114} for an application of this result to construct the degeneration of $\PP^2$ to the weighted projective plane $\PP(1,1,4)$.

\subsection{Outline of the article} 

In \S\ref{sec:preliminaries_toric_geometry} we discuss Cox coordinates on toric varieties and we study polarised projective toric varieties.
In \S\ref{sec:deformations_affine_mavlyutov} we recall Mavlyutov's construction of deformations of affine toric varieties.
In \S\ref{sec:deformations of toric affine pairs} we construct deformations of affine toric pairs and we prove Theorem~\ref{thm:deformations_affine_toric_pairs_easy}.
In \S\ref{sec:deformations_projective_toric} we construct deformations of projective toric varieties.
In \S\ref{sec:deformations_projective_toric_pairs} we construct deformations of projective toric pairs and we prove Theorem~\ref{thm:deformations_projective_toric_pairs_easy}.
In \S\ref{sec:mutations_induce_deformations} we recall the notion of mutation between Fano polytopes and we prove Theorem~\ref{thm:mutations_induce_deformations_easy}.

\subsection{Notation and conventions}

The sets of non-negative or positive integers are denoted by $\NN := \{ 0, 1, 2, 3, \dots \}$ and $\NN^+ := \{ 1, 2, 3, \dots \}$, respectively.
The symbol $\CC$ stands for an algebraically closed field of characteristic zero.

A \emph{lattice} is a finitely generated free abelian group. The letters $N, N_0, \tilde{N}, \tilde{N}_0$ stand for lattices and $M, M_0, \tilde{M}, \tilde{M}_0$ for their duals, e.g.\ $M = \Hom_\ZZ (N, \ZZ)$. We set $N_\RR := N \otimes_\ZZ \RR$ and $M_\RR := M \otimes_\ZZ \RR$. The perfect pairing $M \times N \to \ZZ$ and its extension $M_\RR \times N_\RR \to \RR$ are denoted by the symbol $\langle \cdot, \cdot \rangle$.

In a real vector space $V$ of finite dimension, a \emph{cone} is a non-empty subset which is closed under sum and multiplication by non-negative real numbers. The conical hull $\cone{S}$ of a subset $S \subseteq V$ is the smallest cone containing $S$, i.e.\ the set made up of $\lambda_1 s_1 + \cdots + \lambda_k s_k$, as $k \in \NN$, $\lambda_i \geq 0$, and $s_i \in S$. A subset of $V$ is called a \emph{polyhedral cone} if it coincides with $\cone{S}$ for some finite subset $S \subseteq V$, or equivalently it is the intersection of a finite number of closed halfspaces passing through the origin. The convex hull of a subset $S \subseteq V$ is denoted by $\conv{S}$.
A \emph{polyhedron} is the intersection of a finite number of closed halfspaces, so it is always convex and closed. A compact polyhedron is called \emph{polytope}. If $Q$ is a polyhedron, $\vertices(Q)$ denotes the set of vertices of $Q$ and $\rec(Q)$ is its recession cone, i.e.\ the cone of the unbounded directions of $Q$. 
If $Q_1$ and $Q_2$ are polyhedra, then their \emph{Minkowski sum} is $Q_1 + Q_2 := \{ q_1 + q_2 \mid q_1 \in Q_1, q_2 \in Q_2 \}$; in this case we say also that this is a \emph{Minkowski decomposition} of $Q$. If $Q$ is a polyhedron such that $\rec(Q)$ is strongly convex, then $Q = \conv{\vertices(Q)} + \rec(Q)$.
We refer the reader to the book \cite{ziegler} for details.

We assume the standard terminology of commutative algebra and of algebraic geometry. By a ring we always understand a commutative ring with unit.

\subsection*{Acknowledgements} 

Parts of this article appear in the author's Ph.D.\ thesis at Imperial College London. The author wishes to thank his advisor Alessio Corti for suggesting the problem and Thomas Prince for fruitful conversations. Moreover, he is grateful to Alessandro Chiodo, Nathan Ilten, Alexander Kasprzyk, and Richard Thomas for helpful comments on a preliminary version of this manuscript. 
The author was funded by a Roth Studentship from Imperial College London, by Tom Coates' ERC Consolidator Grant~682603, and by Alexander Kasprzyk's EPSRC Fellowship~EP/N022513/1.

\section{Preliminaries on toric geometry} \label{sec:preliminaries_toric_geometry}

\subsection{Cox coordinates}

For generalities about toric varieties we refer the reader to \cite{fulton_toric_varieties} and \cite{cox_toric_varieties}. We firstly treat toric schemes, with split tori, which are defined over arbitrary rings and consider their total coordinate rings.

\begin{remark}[Toric schemes over arbitrary rings] \label{rmk:toric_varieties_general_basis}
Let $A$ be a ring, let $N$ be a lattice, and let $\Sigma$ be a fan of strongly convex rational polyhedral cones in $N_\RR$. For every cone $\sigma \in \Sigma$, we consider its dual $\sigma^\vee \subseteq M_\RR$, the semigroup $\sigma^\vee \cap M$, and the semigroup $A$-algebra $A[\sigma^\vee \cap M]$. We denote by $\toric{A}{\Sigma}$ the scheme obtained by gluing the affine schemes $\toric{A}{\sigma} = \Spec A[\sigma^\vee \cap M]$ thanks to the structure of the fan $\Sigma$, as it is customary in toric geometry. One may prove that $\toric{A}{\Sigma}$ is a separated flat scheme of finite presentation over $A$ with relative dimension $\rank N$ and geometrically integral fibres. When $A = \CC$, $\toric{A}{\Sigma} = \toric{\CC}{\Sigma}$ is exactly the toric variety over $\CC$ associated to the fan $\Sigma$ considered in \cite{fulton_toric_varieties, cox_toric_varieties}. 

Now suppose that $N_\RR$ is generated as an $\RR$-vector space by the support $\vert \Sigma \vert$ of $\Sigma$. In other words we assume that $\toric{\CC}{\Sigma}$ has \emph{no torus factors}. Let $\Sigma(1)$ be the set of rays of $\Sigma$. We do not distinguish a ray of $\Sigma$, which is actually a $1$-dimensional cone of $\Sigma$, from its primitive generator, which is actually the lattice point on the ray that is the closest one to the origin. Generalising the definition of Cox coordinates on toric varieties (see \cite{cox_coordinates},\cite[\S{5.2}]{cox_toric_varieties} or \cite{miller_sturmfels}), we say that the polynomial ring $S = A [ x_\rho \mid \rho \in \Sigma(1)]$ is the total coordinate ring of $\toric{A}{\Sigma}$. The variables $x_\rho$ are called \emph{Cox coordinates} or homogeneous coordinates. The $A$-algebra $S$ has a grading with respect to the divisor class group $G_\Sigma = \mathrm{Cl}(\toric{\CC}{\Sigma})$ of the variety $\toric{\CC}{\Sigma}$, which is a quotient of the free abelian group $\ZZ^{\Sigma(1)}$ according to the divisor sequence of $\Sigma$ (see \cite[(5.1.1)]{cox_toric_varieties}):
\begin{equation*}
0 \longrightarrow M \longrightarrow \ZZ^{\Sigma(1)} \longrightarrow G_\Sigma = \mathrm{Cl}(\toric{\CC}{\Sigma}) \longrightarrow 0.
\end{equation*}
For every cone $\sigma \in \Sigma$, setting $x^{\hat{\sigma}} = \prod_{\rho \notin \sigma(1)} x_\rho \in S$, the map defined by
\begin{equation*}
\mathrm{Cox} \colon \chi^u \mapsto x^u = \prod_{\rho \in \Sigma(1)} x_\rho^{\langle u, \rho \rangle},
\end{equation*}
where $u \in \sigma^\vee \cap M$ and $\chi^u$ is the corresponding element in $A[\sigma^\vee \cap M]$,
induces a ring isomorphism
\begin{equation*}
A [\sigma^\vee \cap M] \simeq S_{(x^{\hat{\sigma}})} \subseteq S_{x^{\hat{\sigma}}},
\end{equation*}
where $S_{x^{\hat{\sigma}}}$ is the localization of $S$ obtained by inverting the element $x^{\hat{\sigma}}$ and $S_{(x^{\hat{\sigma}})}$ is the subring of the $S_{x^{\hat{\sigma}}}$ consisting of elements of degree $0$ with respect to the $G_\Sigma$-grading.

Imitating \cite[\S{5.3}]{cox_toric_varieties}, from a $G_\Sigma$-graded $S$-module $E$ one may construct a quasi-coherent sheaf $\tilde{E}$ on $\toric{A}{\Sigma}$ such that, for every cone $\sigma \in \Sigma$, the sections of $\tilde{E}$ over $\toric{A}{\sigma}$ are the elements of $E_{(x^{\hat{\sigma}})}$, i.e.\ the elements of degree $0$ in the localization $E_{x^{\hat{\sigma}}}$. The assignment $E \mapsto \tilde{E}$ is sometimes called \emph{sheafification} and is an exact functor from the category of $G_\Sigma$-graded $S$-modules to the category of quasi-coherent sheaves on $\toric{A}{\Sigma}$. In particular, the sheafification of a $G_\Sigma$-homogeneous ideal $J$ of $S$ induces a closed subscheme of $\toric{A}{\Sigma}$, whose structure sheaf is the sheafification of $S/J$. Moreover, if $A$ is noetherian and $E$ is finitely generated graded $S$-module, then $\tilde{E}$ is coherent on $\toric{A}{\Sigma}$. 
\end{remark}

The following lemma gives a sufficient criterion to ensure the flatness of the sheafification of a graded module on a toric scheme.

\begin{lemma} \label{lemma:flatness_sheafification_graded_module}
Let $N$ be a lattice and let $\Sigma$ be a fan of strongly convex rational polyhedral cones in $N_\RR$ such that $N_\RR$ is generated by $\vert \Sigma \vert$ as $\RR$-vector space. Let $A$ be a ring and let $\toric{A}{\Sigma}$ be the $A$-scheme constructed in Remark \ref{rmk:toric_varieties_general_basis}. Let $S$ be the total coordinate ring of $\toric{A}{\Sigma}$ and let $E$ be a graded $S$-module.
If $E$ is flat as an $A$-module, then $\tilde{E} \in \QCoh(\toric{A}{\Sigma})$ is flat over $\Spec A$.
\end{lemma}

\begin{proof}
It is enough to show that $E_{(x^{\hat{\sigma}})}$ is flat over $A$, for every cone $\sigma \in \Sigma$. The localisation $E_{x^{\hat{\sigma}}}$ is a $G_\Sigma$-graded flat $A$-module and the homogeneous localisation $E_{(x^{\hat{\sigma}})}$ is its degree zero part. Therefore, $E_{(x^{\hat{\sigma}})}$ is a direct summand of $E_{x^{\hat{\sigma}}}$ as $A$-modules and is flat over $A$.
\end{proof}

\subsection{Polarised projective toric varieties}
\label{sec:polarised_projective_toric_varieties}

Now we discuss projective toric varieties  $X$ polarised by an ample $\QQ$-Cartier $\QQ$-divisor which is supported on the toric boundary $\partial X$. This section is not necessary for \S\ref{sec:deformations_affine_mavlyutov} and \S\ref{sec:deformations of toric affine pairs}.

The lemma below is a well known characterisation of polarised projective toric varieties.

\begin{lemma} \label{lemma:polarised_projective_varieties}
If $N$ is a lattice of rank $n$, then the following data are naturally equivalent:
\begin{enumerate}
\item a pair $(X,D)$, where $X$ is a projective normal toric variety over $\CC$ with respect to the torus $T_N = \Spec \CC[M]$ and $D$ is an ample torus-invariant $\QQ$-Cartier $\QQ$-divisor on $X$;
\item a pair $(\Sigma, \phiv)$, where $\Sigma$ is a complete fan in $N$ and $\phiv$ is a \emph{strictly convex rational support function} on $\Sigma$, i.e.\ $\phiv \colon N_\RR \to \RR$ is a continuous function such that
\begin{itemize}
\item for every $\sigma \in \Sigma(n)$, there exists $u_\sigma \in M_\QQ$ such that $\phiv (v) = \langle u_\sigma, v \rangle$ for all $v \in \sigma$;
\item for every $\sigma \in \Sigma(n)$, $\phiv(v) < \langle u_\sigma, v \rangle$ for all $v \in N_\RR \setminus \sigma$;
\end{itemize}
\item a rational polytope $P \subseteq M_\RR$ of dimension $n$;
\item a strictly convex rational polyhedral cone $\tau$ in the lattice $N_0 = N \oplus \ZZ e_0$ such that the dimension of $\tau$ is $n+1$ and $e_0$ is in the interior of $\tau$.
\end{enumerate}
In the setting above there are natural bijective correspondences if in addition we require the following further conditions too:
\begin{enumerate}
\item $D$ is a $\QQ$-Cartier $\ZZ$-divisor on $X$;
\item $\phiv$ takes integer values on the primitive generators of the rays of $\Sigma$;
\item every supporting hyperplane of $P$ contains at least a point of the lattice $M$;
\item the primitive generator of every ray of $\tau$ is of the form $\rho - a e_0$ for some $a \in \ZZ$ and $\rho \in N$ primitive.
\end{enumerate}
Moreover, in the setting above there are natural bijective correspondences if we require the following more restrictive further conditions too:
\begin{enumerate}
\item $D$ is a Cartier divisor on $X$;
\item $\phiv$ is a \emph{strictly convex integral support function} on $\Sigma$, i.e.\ we also require that $u_\sigma \in M$ for every $\sigma \in \Sigma(n)$;
\item $P$ is a lattice polytope;
\item every facet of $\tau$ is contained in a hyperplane of the form $(u + e_0^*)^\perp$ for some $u \in M$.
\end{enumerate}
\end{lemma}

\begin{proof}[Sketch of proof]
The equivalence among (1), (2), and (3) is well known (at least under the additional conditions) and associates the pair $(\Sigma, \phiv)$ to the pair 
$(\toric{\CC}{\Sigma}, D)$, where $D = - \sum_{\rho \in \Sigma(1)} \phiv(\rho) D_\rho$, and to the polytope
\begin{equation*}
P = \bigcap_{\rho \in \Sigma(1)} \left\{ u \in M_\RR \left\vert \langle u, \rho \rangle \geq \phiv(\rho) \right. \right\}.
\end{equation*}
Conversely, $\Sigma$ is the normal fan of $P$ and $\phiv = \min_{u \in P} \langle u, \cdot \rangle$. We refer the reader to \cite[\S{6}]{cox_toric_varieties} for more details.

The equivalence with (4) is as follows: $\tau$ is the convex hull of the graph of the function $-\phiv$, i.e.\ $
\tau = \{ v + k e_0 \in N_\RR \oplus \RR e_0 \mid \phiv (v) + k \geq 0 \}
$, or equivalently the cone with rays $\rho - \phiv(\rho) e_0$ as $\rho \in \Sigma(1)$. Conversely, the cones of $\Sigma$ are precisely the images of the faces of $\tau$ along the projection $N \oplus \ZZ e_0 \onto N$ and
\[
P = \tau^\vee \cap e_0 \inv (1) = \{ u \in M_\RR \mid u + e_0^* \in \tau^\vee \}. \qedhere
\]
\end{proof}

The following lemma, which is a reformulation of \cite[Theorem~7.1.13 and Proposition 8.2.11]{cox_toric_varieties}, describes a polarised projective toric variety $X$ as the Proj of an $\NN$-graded ring constructed from the cone $\tau$, where $\tau$ is the cone as in Lemma \ref{lemma:polarised_projective_varieties}. It also gives a description of the toric boundary.

\begin{lemma} \label{lemma:polarised_toric_proj}
Let $N$ be a lattice of rank $n$, let $\tau$ be a $(n+1)$-dimensional strongly convex rational polyhedral cone in the lattice $N_0 = N \oplus \ZZ e_0$ such that $e_0 \in \interior(\tau)$, and let $(X,D)$ be the pair associated to $\tau$ via Lemma \ref{lemma:polarised_projective_varieties}. Consider the ideal
\begin{equation} \label{eq:ideal_boundary_polarised_proj_L}
L = \bigoplus_{u + l e_0^* \in \interior(\tau^\vee) \cap M_0 } \CC \chi^{u + l e_0^*} \subseteq \CC[\tau^\vee \cap M_0],
\end{equation}
which is the ideal of the toric boundary of the affine toric variety $\Spec \CC[\tau^\vee \cap M_0]$.

Then $X = \Proj \CC[\tau^\vee \cap M_0]$ and its toric boundary is $\partial X = \Proj \CC[\tau^\vee \cap M_0] / L$, where $\CC[\tau^\vee \cap M_0]$ is $\NN$-graded via $e_0 \in N_0$.
\end{lemma}

\begin{proof}[Sketch of proof]
The $\NN$-grading on $\CC[\tau^\vee \cap M_0]$ is such that the degree of $\chi^{u + l e_0^*}$ is $l$ for every $u \in M$ such that $u + l e_0^* \in \tau^\vee \cap M_0$. It is clear that $L$ is a homogeneous ideal.

Let $\Sigma$ be the fan of $X$ and let $\phiv$ be the support function associated to $D$ as in Lemma~\ref{lemma:polarised_projective_varieties}. There is a bijective correspondence between cones of $\Sigma$ and proper subcones of $\tau$. For any ray $\rho \in \Sigma(1)$, let $\xi_\rho \in \tau(1)$ be the corresponding ray of $\tau$. In other words, $\xi_\rho = b_\rho \rho - a_\rho e_0$ where $b_\rho \in \NN^+$, $a_\rho \in \ZZ$ are such that $\gcd(a_\rho, b_\rho) = 1$ and $\phiv(\rho) = a_\rho / b_\rho$.

Fix an $n$-dimensional cone $\sigma \in \Sigma(n)$. It corresponds to an $n$-dimensional face of $\tau$, namely $F_\sigma = \cone{\xi_\rho \mid \rho \in \sigma(1)}$. Since $D$ is $\QQ$-Cartier, there exist $u_\sigma \in M$ and $h_\sigma \in \NN^+$ such that $F_\sigma$ is contained in the hyperplane $(u_\sigma + h_\sigma e_0^*)^\perp$.
The affine open subscheme $\toric{\CC}{\sigma}$ of the toric variety $X = \toric{\CC}{\Sigma}$ is isomorphic to the affine open subscheme of $\Proj \CC[\tau^\vee \cap M_0]$ defined by the homogeneous element $\chi^{u_\sigma + h_\sigma e_0^*}$ because there is a ring isomorphism
\begin{equation} \label{eq:anello_affine_sigma_isom_localizzazione _omogenea}
\CC [ \tau^\vee \cap M_0]_{(\chi^{u_\sigma + h_\sigma e_0^*})} \overset{\sim}{\longrightarrow} \CC[\sigma^\vee \cap M]
\end{equation}
which is defined by
\begin{equation*}
\frac{\chi^{u + k h_\sigma e_0^*}}{(\chi^{u_\sigma + h_\sigma e_0^*})^k} \mapsto \chi^{u - k u_\sigma}
\end{equation*}
for any $u \in M$, $k \in \NN$ such that $u + k h_\sigma  e_0^* \in \tau^\vee \cap M_0$. This shows that $X = \Proj \CC[\tau^\vee \cap M_0]$.

In order to prove $\partial X = \Proj \CC[\tau^\vee \cap M_0]/L$, we have to check that, for every cone $\sigma \in \Sigma(n)$, the homogeneous localisation $L_{(\chi^{u_\sigma + h_\sigma e_0^*})}$ coincides with the ideal of the toric boundary of $\toric{\CC}{\sigma}$ under the ring isomorphism \eqref{eq:anello_affine_sigma_isom_localizzazione _omogenea}. So, let us fix a cone $\sigma \in \Sigma(n)$ and elements $u\in M$ and $k \in \NN$ such that $u - k u_\sigma \in \sigma^\vee$. The element $\chi^{u+k h_\sigma e_0^*} / (\chi^{u_\sigma + h_\sigma e_0^*})^k \in \CC[\tau^\vee \cap M_0]_{(\chi^{u_\sigma + h_\sigma e_0^*})}$ lies in $L_{(\chi^{u_\sigma + h_\sigma e_0^*})}$ if and only if $\chi^{u+k h_\sigma e_0^*} \in \left( L \colon (\chi^{u_\sigma + h_\sigma e_0^*})^\infty \right)$, i.e.\ there exists $m \in \NN$ such that
\[
u + k h_\sigma e_0^* + m (u_\sigma + h_\sigma e_0^*) \in \interior(\tau^\vee).
\]
In order to check this we need to pair this vector of $M_0$ with the rays of $\tau$, i.e.\ $\xi_\rho = b_\rho \rho - a_\rho e_0$ as $\rho \in \Sigma(1)$. We distinguish two cases:
\begin{itemize}
\item the ray $\rho$ lies in $\sigma$; then $b_\rho \langle u_\sigma, \rho \rangle - a_\rho h_\sigma = 0$; for any $m \in \NN$ we have
\[
\langle u + k h_\sigma e_0^* + m (u_\sigma + h_\sigma e_0^*) , b_\rho \rho - a_\rho e_0 \rangle = b_\rho \langle u - k u_\sigma, \rho \rangle,
\]
which is positive if and only if $\langle u - k u_\sigma, \rho \rangle > 0$;
\item the ray $\rho$ does not lie in $\sigma$; then $b_\rho \langle u_\sigma, \rho \rangle - a_\rho h_\sigma > 0$; then
\begin{gather*}
\langle u + k h_\sigma e_0^* + m (u_\sigma + h_\sigma e_0^*) , b_\rho \rho - a_\rho e_0 \rangle = \\
= b_\rho \langle u, \rho \rangle - a_\rho k h_\sigma + m (b_\rho \langle u_\sigma, \rho \rangle - a_\rho h_\sigma)
\end{gather*}
is positive for $m$ big enough.
\end{itemize}
This shows that the element $\chi^{u+k h_\sigma e_0^*} / (\chi^{u_\sigma + h_\sigma e_0^*})^k$ lies in the homogeneous localisation of $L$ if and only if for every $\rho \in \sigma(1)$ we have $\langle u - k u_\sigma, \rho \rangle > 0$, or equivalently if $u - k u_\sigma$ lies in the ideal of the toric boundary of $\toric{\CC}{\sigma}$.
\end{proof}

In the following lemma we compare the homogeneous coordinate rings of a polarised toric variety and of its affine cone. We deduce an alternate description of closed subschemes of a polarised toric variety.

\begin{lemma} \label{lemma:cox_coordinates_proj_toric_and_affine_cone}
Let $N$ be a lattice of rank $n$, let $\tau$ be a $(n+1)$-dimensional strongly convex rational polyhedral cone in the lattice $N_0 = N \oplus \ZZ e_0$ such that $e_0 \in \interior(\tau)$, and let $(X,D)$ and $(\Sigma, \phiv)$ be the pairs associated to $\tau$ via Lemma \ref{lemma:polarised_projective_varieties}.
Consider the affine toric variety $C = \Spec \CC[\tau^\vee \cap M_0]$. Let $S_X$ and $S_C$ be the homogeneous coordinate rings of $X$ and $C$, respectively.

For every ray $\rho \in \Sigma(1)$, let $\xi_\rho = b_\rho \rho - a_\rho e_0 \in \tau(1)$ be the corresponding ray of $\tau$, where $\phiv(\rho) = a_\rho / b_\rho$ for $a_\rho \in \ZZ$ and $b_\rho \in \NN^+$ such that $\gcd(a_\rho, b_\rho) = 1$.
Consider the ring homomorphism $S_X \to S_C$ given by $x_\rho \mapsto (x_{\xi_\rho})^{b_\rho}$.

Let $J_X$ be a $G_\Sigma$-homogeneous ideal in $S_X$ and let $H \subseteq \CC[\tau^\vee \cap M_0] \simeq (S_C)_0$ be the degree zero part of the ideal $J_X S_C \subseteq S_C$. If $H$ is homogeneous with respect to the $\NN$-grading of $\CC[\tau^\vee \cap M_0]$, then the closed subscheme of $X$ defined by the ideal $J_X$ coincides with $\Proj \CC[\tau^\vee \cap M_0] / H$.
\end{lemma}

\begin{proof}
Consider the commutative diagram
\[
\xymatrix{
\ZZ^{\tau(1)} \ar[d]^b \ar[r]^{r_\tau} & N_0 \ar@{->>}[d]^{\mathrm{pr}} \\
\ZZ^{\Sigma(1)} \ar[r]^{r_\Sigma} & N
}
\]
where $r_\Sigma$ is the ray map of $X$, $r_\tau$ is the ray map of $C$, $\mathrm{pr}$ is the natural projection, and $b$ is the diagonal matrix with entries $b_\rho$. Consider the dual maps $r_\Sigma^*$ and $r_\tau^*$ and the following commutative diagram with exact rows, where $G_\Sigma$ is the divisor class group of $X$ and $G_\tau$ is the divisor class group of $C$.
\[
\xymatrix{
0 \ar[r] & M \ar@{^{(}->}[d] \ar[r]^{r_\Sigma^*} & \ZZ^{\Sigma(1)} \ar[d]^b \ar[r] & G_\Sigma \ar[d] \ar[r] & 0 \\
0 \ar[r] & M_0 \ar[r]^{r_\tau^*} & \ZZ^{\tau(1)} \ar[r] & G_\tau \ar[r] & 0
}
\] 
The ring homomorphism $S_X \to S_C$ is homogeneous with respect to the group homomorphism $G_\Sigma \to G_\tau$. In particular, the ideal $J_X S_C \subseteq S_C$ is $G_\tau$-homogeneous.

Fix a full dimensional cone $\sigma \in \Sigma(n)$ and let $u_\sigma \in M$ and $h_\sigma \in \NN^+$ be such that the hyperplane $(u_\sigma + h_\sigma e_0^*)^\perp$ contains the corresponding face $F_\sigma$ of $\tau$, as in the proof of Lemma~\ref{lemma:polarised_toric_proj}. We set $\bar{u}_\sigma = u_\sigma + h_\sigma e_0^* \in M_0$ for brevity. We have to show that the ideal $(J_X)_{(x^{\hat{\sigma}})} \subseteq (S_X)_{(x^{\hat{\sigma}})} \simeq \CC[\sigma^\vee \cap M]$ is mapped to $H_{(\chi^{\bar{u}_\sigma})}$ via the isomorphism \eqref{eq:anello_affine_sigma_isom_localizzazione _omogenea}.

Since $\bar{u}_\sigma$ is zero on the face $F_\sigma$ and strictly positive on $\tau \setminus F_\sigma$, a Cox coordinate $x_\xi$ of $C$ appear in the monomial $x^{\bar{u}_\sigma} \in S_C$ if and only if $\xi \notin F_\sigma$. This implies that there is a ring homomorphism
\begin{equation} \label{eq:localised_S_X_to_S_C}
(S_X)_{x^{\hat{\sigma}}} \longrightarrow (S_C)_{x^{\bar{u}_\sigma}}
\end{equation}
that is the localisation of $S_X \to S_C$ defined above. At this point it is not difficult to show that there is a commutative diagram of rings
\begin{equation*}
\xymatrix{
& \CC[\tau^\vee \cap M_0] \ar@{=}[r]^{\qquad \mathrm{Cox}_\tau} \ar[d] & (S_C)_0 \ar[d] \ar@{^{(}->}[r] & S_C \ar[d] \\
\CC[\tau^\vee \cap M_0]_{(\chi^{\bar{u}_\sigma})} \ar@{=}_{\eqref{eq:anello_affine_sigma_isom_localizzazione _omogenea}}[d] \ar@{^{(}->}[r] & \CC[\tau^\vee \cap M_0]_{\chi^{\bar{u}_\sigma}} \ar@{=}[r]^{\ \ \ \ \mathrm{Cox}_{F_\sigma}} & (S_C)_{(x^{\bar{u}_\sigma})} \ar@{^{(}->}[r] & (S_C)_{x^{\bar{u}_\sigma}} \\
\CC[\sigma^\vee \cap M] \ar@{=}[rr]^{\mathrm{Cox}_\sigma} & & (S_X)_{(x^{\hat{\sigma}})} \ar@{^{(}->}[r] & (S_X)_{x^{\hat{\sigma}}} \ar[u]^{\eqref{eq:localised_S_X_to_S_C}} & S_X \ar[l] \ar@/_/[luu]
}
\end{equation*}
where the equality symbols stand for isomorphisms. 
Now consider the ideal $K = J_X (S_C)_{x^{\bar{u}_\sigma}} \subseteq (S_C)_{x^{\bar{u}_\sigma}}$. 

Since $S_C$ is a finite free $S_X$-module, $S_C$ is faithfully flat over $S_X$. Therefore, also the localised homomorphism \eqref{eq:localised_S_X_to_S_C} is faithfully flat. By \cite[Theorem 7.5(ii)]{matsumura} the contraction of $K$ to $(S_X)_{x^{\hat{\sigma}}}$ is the extension of $J_X$. This implies that $(J_X)_{(x^{\hat{\sigma}})}$ is the contraction of $K$ to $(S_X)_{(x^{\hat{\sigma}})}$ along the homomorphisms in the diagram above.

On the other hand, it is clear that $K$ is the extension of $J_X S_C$ to $(S_C)_{x^{\bar{u}_\sigma }}$. Since $x^{\bar{u}_\sigma}$ has degree zero with respect to the $G_\tau$-grading of $S_C$, it is not difficult to check that the extension of $H = (J_X S_C) \cap (S_C)_0$ to $(S_C)_{(x^{\bar{u}_\sigma})} \simeq \CC[\tau^\vee \cap M_0]_{\chi^{\bar{u}_\sigma}}$ is the contraction of $K$. It follows that the ideal $H_{(\chi^{\bar{u}_\sigma})}$ is the contraction of $K$ to $\CC[\tau^\vee \cap M_0]_{(\chi^{\bar{u}_\sigma})}$.

Since the two ideals that must be checked to coincide are both contractions of the same ideal $K$, we are done.
\end{proof}

\section{Deformations of affine toric varieties after A.~Mavlyutov} \label{sec:deformations_affine_mavlyutov}

In this section we recall the work \cite{mavlyutov} by Anvar Mavlyutov on the deformations of affine toric varieties. We have rewritten a detailed proof, as it will be useful for our generalisations, and we have taken this opportunity to fill in details missing from Mavlyutov's original paper. In so doing we have reformulated many of his statements in terms of deformation datum.

\begin{definition} \label{def:deformation_datum}
Let $N$ be a lattice and $\sigma \subseteq N_\RR$ be a strongly convex rational polyhedral cone with dimension $\rank N$. A \emph{deformation datum for $(N,\sigma)$} is a tuple $(Q, Q_0, Q_1, \dots, Q_k, w)$ where $w \in M$ and $Q$, $Q_0$, $Q_1$, \dots, $Q_k$ are non-empty rational polyhedra in $N_\RR$ such that the following conditions are satisfied:
\begin{enumerate}
\item[(i)] $Q \subseteq \sigma$;
\item[(ii)] $0 \notin Q$;
\item[(iii)] $Q = Q_0 + Q_1 + \cdots + Q_k$;
\item[(iv')] for every vertex $v \in \vertices(Q)$, there exist vertices $v_0 \in \vertices(Q_0)$, $v_1 \in \vertices(Q_1)$, \dots, $v_k \in \vertices(Q_k)$ such that $v = v_0 + v_1 + \cdots + v_k$ and
\begin{equation*}
\# \left\{ i \in \{0,1,\dots,k \} \mid v_i \notin N \right\} \leq 1;
\end{equation*}
\item[(v)] the minimum of $w$ on $Q$ exists and is not smaller than $-1$;
\item[(vi)] every vertex of the polyhedron $\sigma \cap \{ n \in N_\RR \vert \langle w, n \rangle = -1 \}$ is contained in $\RR^+ \cdot Q$.
\end{enumerate}
A \emph{$\partial$-deformation datum for $(N,\sigma)$} is a deformation datum $(Q, Q_0, Q_1, \dots, Q_k, w)$ for $(N,\sigma)$ such that the following further condition is satisfied:
\begin{enumerate}
\item[(iv)] $Q_1$, \dots, $Q_k$ are lattice polyhedra.
\end{enumerate}
\end{definition}

It is immediate to see that (iv) implies (iv').

\begin{notation} \label{not:tilde_N_tilde_sigma_tilde_w}
If $(Q, Q_0, Q_1, \dots, Q_k, w)$ is a deformation datum for $(N,\sigma)$, then we set
\begin{align*}
\tilde{N} &:= N \oplus \ZZ e_1 \oplus \cdots \oplus \ZZ e_k \\
\tilde{\sigma} &:= \cone{\sigma, Q_0-e_1-\cdots-e_k, Q_1 + e_1, \dots, Q_k + e_k} \subseteq \tilde{N}_\RR \\
\tilde{w} &:= w - \sum_{i=1}^k \left\lfloor \min_{Q_i} w \right\rfloor e_i^* \in \tilde{M}
\end{align*}
\end{notation}

\begin{remark}
If $N$ is a lattice, $\sigma \subseteq N_\RR$ is a $(\rank N)$-dimensional strongly convex rational polyhedral cone,  $w \in M$, and $Q = \conv{ \vertices( \sigma \cap \{v \in N_\RR \vert \langle w, v \rangle = -1 \})} = Q_0 + Q_1 + \cdots + Q_k$ where $Q_0$ is a rational polytope and $Q_1, \dots, Q_k$ are lattice polytopes, then $(Q, Q_0, Q_1, \dots, Q_k,w)$ is a $\partial$-deformation datum for $(N,\sigma)$. Moreover, if in addition $Q_i \subseteq  \{v \in N_\RR \vert \langle w, v \rangle = 0 \} =: w^\perp$ for $i=1, \dots, k$, then $\tilde{w} = w$.
\end{remark}

\begin{lemma} \label{lemma:sigma_tilde_is_full_dimensional}
Let $N$ be a lattice of rank $n$, let $\sigma \subseteq N_\RR$ be a strongly convex rational polyhedral cone of dimension $n$, let $(Q,Q_0,Q_1, \dots, Q_k,w)$ be a deformation datum for $(N,\sigma)$, and let $\tilde{N}$ and $\tilde{\sigma}$ be as in Notation~\ref{not:tilde_N_tilde_sigma_tilde_w}. Then $\tilde{\sigma}$ is a strongly convex rational polyhedral cone in $\tilde{N}$ of dimension $n+k$ such that $\sigma = \tilde{\sigma} \cap N_\RR$.
\end{lemma}

We will postpone the proof of Lemma~\ref{lemma:sigma_tilde_is_full_dimensional} to page \pageref{proof:lemma_sigma_tilde_full_dimensional}.

\begin{theorem}[Mavlyutov \cite{mavlyutov}] \label{thm:mavlyutov}
Let $N$ be a lattice, let $\sigma \subseteq N_\RR$ be a strongly convex rational polyhedral cone with dimension $\rank N$, let $(Q,Q_0,Q_1, \dots, Q_k,w)$ be a deformation datum for $(N,\sigma)$, and let $\tilde{N}$, $\tilde{\sigma}$ and $\tilde{w}$ be as in Notation~\ref{not:tilde_N_tilde_sigma_tilde_w}.
Let $X$ be the affine toric variety associated to $\sigma$ and let $\tilde{X}$ be the affine toric variety associated to $\tilde{\sigma}$.

\smallskip

{\rm (A)} Then the toric morphism $X \to \tilde{X}$, induced by the inclusion $N \into \tilde{N}$, is a closed embedding and identifies $X$ with the closed subscheme of $\tilde{X}$ associated to the homogeneous ideal generated by the following binomials in the Cox coordinates of $\tilde{X}$:
\begin{equation*}
\prod_{\substack{\xi \in \tilde{\sigma}(1) \colon \\ \langle e_i^*, \xi \rangle > 0}} x_\xi^{\langle e_i^*, \xi \rangle} - 
\prod_{\substack{\xi \in \tilde{\sigma}(1) \colon \\ \langle e_i^*, \xi \rangle < 0}} x_\xi^{-\langle e_i^*, \xi \rangle}
\end{equation*}
for $i = 1, \dots, k$. Moreover, these binomials form a regular sequence.

\smallskip

{\rm (B)} 
Let $t_1, \dots, t_k$ be the standard coordinates on $\AA^k_\CC$. Consider the closed subscheme $\cX$ of $\tilde{X} \times_{\Spec \CC} \AA^k_\CC = \toric{\CC[t_1, \dots, t_k]}{\tilde{\sigma}}$ defined by the homogeneous ideal generated by the following trinomials in Cox coordinates:
\begin{equation} \label{eq:trinomials}
\prod_{\substack{\xi \in \tilde{\sigma}(1) \colon \\ \langle e_i^*, \xi \rangle > 0}} x_\xi^{\langle e_i^*, \xi \rangle} - 
\prod_{\substack{\xi \in \tilde{\sigma}(1) \colon \\ \langle e_i^*, \xi \rangle < 0}} x_\xi^{-\langle e_i^*, \xi \rangle} -
t_i \prod_{\xi \in \tilde{\sigma}(1)} x_\xi^{\langle \tilde{w}, \xi \rangle} \prod_{\substack{\xi \in \tilde{\sigma}(1) \colon \\ \langle e_i^*, \xi \rangle < 0}} x_\xi^{-\langle e_i^*, \xi \rangle}
\end{equation}
for $i=1, \dots, k$. Then the morphism $\cX \to \AA^k_\CC$ induces a formal deformation of $X$ over $\CC [ \![ t_1, \dots, t_k ] \! ]$.
\end{theorem}

\begin{remark} \label{rmk:formally_flat}
We will clarify what we mean when we say that the aforementioned closed subscheme induces a \emph{formal deformation} of $\toric{\CC}{\sigma}$ over $\CC [ \![ t_1, \dots, t_k ] \! ]$. By (A) the fibre of $\cX \to \AA^k_\CC$ over the origin of $\AA^k_\CC$ is $X$. We do not know if $\cX \to \AA^k_\CC$ is a flat morphism, but it is ``formally flat'' over the origin in the following sense: for every $(t_1, \dots, t_k)$-primary ideal $\frakq$ of $\CC[t_1, \dots, t_k]$, the fibre product $\cX \times_{\AA^k_\CC} \Spec \CC[t_1, \dots, t_k] / \frakq$ is flat over $\Spec \CC[t_1, \dots, t_k] / \frakq$. Since the inverse limit of the rings $\CC[t_1, \dots, t_k] / \frakq$ is $\CC[ \! [t_1, \dots, t_k ] \! ]$, we say that we have a formal deformation over $\CC[ \! [t_1, \dots, t_k ] \! ]$ by using  \`a la Schlessinger terminology.

As we will see in \S\ref{sec:deformations_projective_toric}, if we had been dealing with deformations of complete varieties there would have been no need to specify the adverb ``formally'' thanks to Lemma~\ref{lemma:flatness_from_formal_to_local_proper}
\end{remark}

\begin{remark} \label{rmk:homogeneous}
Here we explain the meaning of the adjective \emph{homogeneous} in the title of this paper.
Let us assume that we are using notation from Theorem~\ref{thm:mavlyutov}. The torus $T_N = \Spec \CC[M]$ acts on the affine toric variety $X$ and consequently on $\rT^1_X = \Ext^1(\Omega_X, \cO_X)$, which is the tangent space of the deformation functor of $X$. Therefore $\rT^1_X$ is an $M$-graded vector space over the field $\CC$. Every formal deformation of $X$ over a complete local noetherian $\CC$-algebra $(A,\frakm_A)$ with residue field $\CC$ induces a $\CC$-linear map
\[
\left( \frakm_A / \frakm_A^2 \right)^\vee \to \rT^1_X
\]
which is called the Kodaira--Spencer map of the deformation (see \cite[\S6.1.2]{talpo_vistoli_deformation}). By \cite[Theorem~2.14]{mavlyutov} the image of the Kodaira--Spencer map of the deformation of $X$ constructed in Theorem~\ref{thm:mavlyutov} is contained in $\rT^1_X(w)$, which is the homogeneous component of $\rT^1_X$ with degree $w \in M$.
\end{remark}

The rest of this section is devoted to the proof of Theorem \ref{thm:mavlyutov} and relies entirely on \cite{mavlyutov}.

The following lemma is a very particular case of a result by K.~G.~Fischer and J.~Shapiro \cite{fischer_shapiro} that gives a necessary and sufficient criterion for a sequence of binomials to be a regular sequence. For every $a \in \ZZ$, define $a^+ := \max \{ a ,0 \}$ and $a^- := \max \{ -a, 0 \}$.

\begin{lemma}[Fischer--Shapiro~\cite{fischer_shapiro}] \label{lemma:fischer_shapiro}
Let $\cM = \left( a_{ij} \right)_{1 \leq i \leq k, 1 \leq j \leq n}$ be a $k \times n$ matrix with entries in $\ZZ$. For every $i = 1, \dots, k$, consider the binomial
\begin{equation*}
f_i = \prod_{j=1}^n x_j^{a_{ij}^+} - \prod_{j=1}^n x_j^{a_{ij}^-} \in \CC[x_1, \dots, x_n].
\end{equation*}
If the rank of $\cM$ is $k$ and every column of $\cM$ has at most one positive entry, then $f_1, \dots, f_k$ is a regular sequence in $\CC[x_1, \dots, x_n]$.
\end{lemma}

When we have a cone in a lattice $\tilde{N}$, it is possible to intersect it with a saturated sublattice $N$ of $\tilde{N}$ and get a toric morphism. The following lemma describes the scheme-theoretic image of this toric morphism under some hypotheses. 

\begin{lemma} \label{lemma:cox_equations_slicing_cone}
Let $N$ be a lattice and let $\tilde{N} = N \oplus \ZZ^k$. Denote by $e_1, \dots, e_k$ the standard basis of $\ZZ^k$. Let $\tilde{\sigma} \subseteq \tilde{N}_\RR$ be a $(\rank \tilde{N})$-dimensional strongly convex rational polyhedral cone that satisfies the following condition: the $\ZZ^k$-component of every ray of $\tilde{\sigma}$ has at most one positive entry, i.e.\
\begin{equation} \label{eq:condition_rays_sigma_tilde}
\tilde{\sigma}(1) \subseteq N \times \left(( - \NN)^k \cup \NN^+ e_1 \cup \cdots \cup \NN^+ e_k \right)
\end{equation}

If $\sigma$ is the cone $\tilde{\sigma} \cap N_\RR$ inside $N_\RR$, then the scheme-theoretic image of the toric morphism
$
\toric{\CC}{\sigma} \to \toric{\CC}{\tilde{\sigma}}
$
is the closed subscheme of $\toric{\CC}{\tilde{\sigma}}$ defined by the homogeneous ideal generated by the following binomials in the Cox coordinates of $\toric{\CC}{\tilde{\sigma}}$:
\begin{equation*}
\prod_{\substack{\xi \in \tilde{\sigma}(1) \colon \\ \langle e_i^*, \xi \rangle > 0}} x_\xi^{\langle e_i^*, \xi \rangle} - 
\prod_{\substack{\xi \in \tilde{\sigma}(1) \colon \\ \langle e_i^*, \xi \rangle < 0}} x_\xi^{-\langle e_i^*, \xi \rangle}
\end{equation*}
for $i = 1, \dots, k$. Moreover, these binomials form a regular sequence.
\end{lemma}

\begin{proof}
The toric morphism $
\toric{\CC}{\sigma} \to \toric{\CC}{\tilde{\sigma}}
$ is associated to the ring homomorphism
\begin{equation} \label{eq:ring_homomorphism_sigma_tilde_sigma}
\CC[\tilde{\sigma}^\vee \cap \tilde{M}] \to \CC[\sigma^\vee \cap M]
\end{equation}
that maps $\chi^{\tilde{u}}$ to $\chi^{\phi(\tilde{u})}$, where $\phi \colon \tilde{\sigma}^\vee \cap \tilde{M} \to \sigma^\vee \cap M$ is the semigroup homomorphism given by $u + a_1 e_1^* + \cdots + a_k e_k^* \mapsto u$. Let $I \subseteq \CC[\tilde{\sigma}^\vee \cap \tilde{M}]$ be the kernel of \eqref{eq:ring_homomorphism_sigma_tilde_sigma}. The scheme-theoretic image of $
\toric{\CC}{\sigma} \to \toric{\CC}{\tilde{\sigma}}
$ is the closed subscheme  of $\toric{\CC}{\tilde{\sigma}}$ defined by the ideal $I$.

We consider the Cox ring of $\toric{\CC}{\tilde{\sigma}}$: $S = \CC[x_\xi \mid \xi \in \tilde{\sigma}(1)]$, with its $G_{\tilde{\sigma}}$-grading. Consider the following monomials in Cox coordinates:
\begin{align*}
y_i = \prod_{\substack{\xi \in \tilde{\sigma}(1) \colon \\ \langle e_i^*, \xi \rangle > 0}} x_\xi^{\langle e_i^*, \xi \rangle}
\qquad \text{and} \qquad
z_i = \prod_{\substack{\xi \in \tilde{\sigma}(1) \colon \\ \langle e_i^*, \xi \rangle < 0}} x_\xi^{-\langle e_i^*, \xi \rangle},
\end{align*}
for $i = 1, \dots, k$. Let $J \subseteq S$ be the ideal generated by $y_1 - z_1, \dots, y_k - z_k$. It is obviously homogeneous.
In order to prove the thesis, we need to show that, 
under the Cox isomorphism between $\CC[\tilde{\sigma}^\vee \cap M]$ and $S_0$,
 the ideal $I$ equals the degree zero part of the ideal $J$, i.e.\
\begin{equation}\label{eq:equality_ideals_cox_isomorphism}
\mathrm{Cox}(I) = J \cap S_0.
\end{equation}

\medskip

We now prove the containment $\subseteq$ in \eqref{eq:equality_ideals_cox_isomorphism}.
Since $I$ is the kernel of \eqref{eq:ring_homomorphism_sigma_tilde_sigma}, it is not difficult to show that $I$ is generated by the elements $\chi^r - \chi^s$ whenever $r, s \in \tilde{\sigma}^\vee \cap \tilde{M}$ are such that $\phi(r) = \phi(s)$. So $r - s = \sum_{i=1}^k a_i e_i^*$, for some $a_i \in \ZZ$. Now, for each $i = 1, \dots, k$, consider $a_i^+ \in \NN$ and $a_i^- \in \NN$: we have $a_i^+ a_i^- = 0$ and $a_i = a_i^+ - a_i^-$. Consider the element
\begin{equation*}
q = r - \sum_{i=1}^k a_i^+ e_i^* = s - \sum_{i=1}^k a_i^- e_i^* \in \tilde{M}.
\end{equation*}
Let us show that $q \in \tilde{\sigma}^\vee$. We need to show that $q$ is non-negative on the rays of $\tilde{\sigma}$. By \eqref{eq:condition_rays_sigma_tilde}, we distinguish two cases:
\begin{itemize}
\item $v = n - b_1 e_1 - \cdots - b_k e_k \in \tilde{\sigma}(1)$, for some $n \in N$ and $b_1, \dots, b_k \in \NN$; then $\langle q, v \rangle = \langle r, v \rangle +  \sum_{i=1}^k a_i^+ b_i \geq \langle r, v \rangle \geq 0$.
\item $v = n + b e_i \in \tilde{\sigma}(1)$, for some $n \in N$, $1 \leq i \leq k$ and $b \in \NN^+$; then  $\langle q, v \rangle = \langle r , v \rangle - a_i^+ b = \langle s, v \rangle - a_i^- b$. Since either $a_i^+ = 0$ or $a_i^- = 0$, we have either $\langle q, v \rangle = \langle r, v \rangle \geq 0$ or $\langle q, v \rangle = \langle s, v \rangle \geq 0$.
\end{itemize}
Therefore $\chi^{q} \in \CC[\tilde{\sigma}^\vee \cap \tilde{M}]$.

For each $\xi \in \tilde{\sigma}(1)$ we have
\begin{equation*}
\langle r, \xi \rangle - \sum_{i : \langle e_i^*, \xi \rangle < 0} a_i^+ \langle e_i^*, \xi \rangle = 
\langle q, \xi \rangle + \sum_{i : \langle e_i^*, \xi \rangle > 0} a_i^+ \langle e_i^*, \xi \rangle.
\end{equation*}
Therefore, in the ring $S$ we have the equality
\begin{equation} \label{eq:cox(chi_r)}
\mathrm{Cox}(\chi^r) \cdot \prod_{i=1}^k z_i^{a_i^+} = \mathrm{Cox}(\chi^q) \cdot \prod_{i=1}^k y_i^{a_i^+}.
\end{equation}
By \eqref{eq:condition_rays_sigma_tilde} every Cox variable appearing in $y_1 \cdots y_k$ does not appear in $z_1 \cdots z_k$.
Therefore the monomials $\prod_{i=1}^k y_i^{a_i^+}$ and $\prod_{i=1}^k z_i^{a_i^+}$ are coprime.
From \eqref{eq:cox(chi_r)} we obtain that the monomial $\prod_{i=1}^k y_i^{a_i^+}$ divides $\mathrm{Cox}(\chi^r)$. Therefore there exists a monomial $p \in S$ such that
\begin{align*}
\mathrm{Cox}(\chi^r) = p \cdot \prod_{i=1}^k y_i^{a_i^+} \qquad \text{and} \qquad
\mathrm{Cox}(\chi^q) = p \cdot \prod_{i=1}^k z_i^{a_i^+};
\end{align*}
thus the binomial
\begin{equation*}
\mathrm{Cox}(\chi^r - \chi^q) = p \cdot \left( \prod_{i=1}^k  y_i^{a_i^+} - \prod_{i=1}^k z_i^{a_i^+} \right)
\end{equation*}
is clearly in the ideal $J$. In a completely analogous way we prove that $\mathrm{Cox}(\chi^s - \chi^q)$ is in $J$. Therefore, by taking the difference, we have that $\mathrm{Cox}(\chi^r - \chi^s)$ is in $J$.

\medskip

We now prove the containment $\supseteq$ in \eqref{eq:equality_ideals_cox_isomorphism}. Let $f \in J \cap S_0$. We may write
\begin{equation*}
f = \sum_{i=1}^k f_i (y_i - z_i)
\end{equation*}
for some $f_i \in S$. Let $\beta_i \in G_{\tilde{\sigma}}$ be the degree of $y_i - z_i$. By taking the homogeneous components with respect to the $G_{\tilde{\sigma}}$-grading, we may assume that $f_i$ is homogeneous of degree $- \beta_i$. By decomposing $f_i$ into the sum of its monomials, in order to show the containment $\supseteq$ in \eqref{eq:equality_ideals_cox_isomorphism}, it is enough to show that $p (y_i - z_i) \in \mathrm{Cox}(I)$, whenever $i \in \{ 1, \dots, k \}$ and $p \in S$ is a monomial of degree $- \beta_i$.

Since $p y_i$ and $p z_i$ are monomials of degree $0$ in $S$, there exist $r, s \in \tilde{\sigma}^\vee \cap \tilde{M}$ such that $ p y_i = \mathrm{Cox}(\chi^r)$ and $ p z_i = \mathrm{Cox}(\chi^s)$. Since $p (y_i - z_i) = \mathrm{Cox}(\chi^r - \chi^s)$, we must show that $\phi(r) = \phi(s)$. Let us assume that
\[
p = \prod_{\xi \in \tilde{\sigma}(1)} x_\xi^{b_\xi}.
\]
For each $\xi \in \tilde{\sigma}(1)$, we have that
\begin{align*}
b_\xi + \langle e_i^*, \xi \rangle^+ &= \langle r, \xi \rangle \\
b_\xi + \langle e_i^*, \xi \rangle^- &= \langle s, \xi \rangle,
\end{align*}
therefore $\langle e_i^*, \xi \rangle = \langle r-s, \xi \rangle$. Since $\tilde{\sigma}$ is full dimensional, we have $r-s = e_i^*$; this proves that $\phi(r) = \phi(s)$ and $\chi^r - \chi^s \in I$.

\medskip

Now we prove that $y_1 - z_1, \dots, y_k - z_k$ is a regular sequence. By Lemma \ref{lemma:fischer_shapiro} it is enough to show that the matrix $\cM = \left( \langle e_i^*, \xi \rangle \right)_{1 \leq i \leq k, \ \xi \in \tilde{\sigma}(1)}$ has rank $k$ and every column of $\cM$ has at most one positive entry. The latter condition is satisfied by \eqref{eq:condition_rays_sigma_tilde}.

The linear map associated to the matrix $\cM$ is the composite of the ray map $\rho \colon \ZZ^{\vert \tilde{\sigma}(1) \vert} \to \tilde{N} = N \oplus \ZZ^k$ of $\toric{\CC}{\tilde{\sigma}}$ and the projection $\pi \colon \tilde{N} = N \oplus \ZZ^k \to \ZZ^k$. Since $\tilde{\sigma}$ is full-dimensional, $\rho \otimes_\ZZ \mathrm{id}_\RR$ is surjective. This implies that $(\pi \circ \rho) \otimes_\ZZ \mathrm{id}_\RR$ is surjective and that $\cM$ has rank $k$.
\end{proof}

\begin{proof}[Proof of Lemma~\ref{lemma:sigma_tilde_is_full_dimensional}] \label{proof:lemma_sigma_tilde_full_dimensional}
By (iii) and (i) in Definition~\ref{def:deformation_datum} we see that $\rec(Q_i) \subseteq \rec(Q) \subseteq \sigma$ for every $i = 0, 1, \dots, k$. In particular, $\rec(Q_i)$ is strongly convex; so $Q_i = \conv{\vertices(Q_i)} + \rec(Q_i)$. We have that
\begin{equation*}
\tilde{\sigma} = \cone{ \sigma, \vertices(Q_0) - e_1 - \cdots - e_k, \vertices(Q_1) + e_1, \dots, \vertices(Q_k) + e_k }.
\end{equation*}
This implies that the cone $\tilde{\sigma}$ is a rational convex polyhedral cone in $\tilde{N}$. Moreover, the rays of $\tilde{\sigma}$ are among the following rays:
\begin{itemize}
\item rays passing through the vertices of $Q_0 - e_1 - \cdots -e_k$;
\item rays passing through the vertices of $Q_i + e_i$, as $ i = 1, \dots, k$;
\item rays of $\sigma$ that are not in the cone generated by the previous rays.
\end{itemize}

Now we prove that $\sigma = \tilde{\sigma} \cap N_\RR$. The containment $\subseteq$ is obvious. We need to show the containment $\supseteq$. Let $\tilde{v} \in \tilde{\sigma} \cap N_\RR$. By the convexity of $Q_0, Q_1, \dots, Q_k$, which implies that $\cone{Q_i + e_i} = \RR_{\geq 0} (Q_i + e_i)$ and an analogous statement for $Q_0$, we may assume that
\begin{align*}
\tilde{v} &= v + \lambda_0 (q_0 - e_1 - \cdots - e_k) + \lambda_1 (q_1 + e_1) + \cdots + \lambda_k (q_k + e_k) \\
&= v + \lambda_0 q_0 + \lambda_1 q_1 + \cdots + \lambda_k q_k + (\lambda_1 - \lambda_0) e_1 + \cdots + (\lambda_k - \lambda_0) e_k
\end{align*}
for some $v \in \sigma$, $q_i \in Q_i$ and $\lambda_i \geq 0$. Since $\tilde{v} \in N_\RR$, $\lambda_0 = \lambda_i$ for every $i$. Therefore $\tilde{v} = v + \lambda_0 (q_0 + q_1 + \cdots + q_k)$. By (iii) and (i), $q_0 + q_1 + \cdots + q_k \in Q \subseteq \sigma$ and we conclude that $\tilde{v} \in \sigma$.

Now we show that $\tilde{\sigma}$ is strongly convex.
Since $\sigma$ is strongly convex and $0 \notin Q$ by (ii), we may find $u \in \interior(\sigma^\vee)$ such that $\min_Q u > 0$.
Since the recession cones of $Q, Q_0, Q_1, \dots, Q_k$ are contained in $\sigma$, the minimum of $u$ on each of these polyhedra exists. Consider
\begin{equation*}
\tilde{u} = u - \sum_{i=1}^k \min_{Q_i} u   \ e_i^* + \frac{1}{k+1} \min_Q u \ \sum_{i=1}^k e_i^* \in \tilde{M}_\RR
\end{equation*}
In order to show that $\tilde{\sigma}$ is strictly convex, we prove that $\tilde{u}$ is positive on the rays of $\tilde{\sigma}$. We may distinguish three cases as follows:
\begin{itemize}
\item the ray passes through $v - e_1 - \cdots - e_k$, for some $v \in \vertices(Q_0)$; then
\begin{align*}
\langle \tilde{u}, v - e_1 - \cdots - e_k \rangle &= \langle u, v \rangle + \sum_{i=1}^k \min_{Q_i} u - \frac{k}{k+1} \min_Q u \\
&\geq \min_{Q_0} u + \min_{Q_1 + \cdots + Q_k} u - \frac{k}{k+1} \min_Q u \\
&= \min_Q u - \frac{k}{k+1} \min_Q u \\
&= \frac{1}{k+1} \min_Q u > 0;
\end{align*} 
\item the ray passes through $v + e_i$, for some $v \in \vertices(Q_i)$ and $1 \leq i \leq k$; then
\begin{equation*}
\langle \tilde{u}, v + e_i \rangle = \langle u, v \rangle - \min_{Q_i} u + \frac{1}{k+1} \min_Q u \geq \frac{1}{k+1} \min_Q u > 0.
\end{equation*}
\item the ray is a ray of $\sigma$ through $v \in N \setminus \{ 0 \}$; then $\langle \tilde{u}, v \rangle = \langle u, v \rangle > 0$, because $u \in \interior(\sigma^\vee)$;
\end{itemize}
This concludes the proof of the strong convexity of $\tilde{\sigma}$.

We now show that $\tilde{\sigma}$ has dimension $\rank \tilde{N}$. Equivalently we see that zero is the unique linear functional on $\tilde{N}$ that vanishes over $\tilde{\sigma}$. Let $\tilde{u} = u + \sum_{i=1}^k a_i e_i^* \in \tilde{M}$ be such that it vanishes over $\tilde{\sigma}$. In particular it vanishes over $\sigma$, hence $u = 0$ because $\sigma$ is full-dimensional. By evaluating $\tilde{u}$ on $Q_i + e_i$ we see that $a_i$ must be zero. This implies that $\tilde{u} = 0$.

This concludes the proof of Lemma~\ref{lemma:sigma_tilde_is_full_dimensional}.
\end{proof}

\begin{proof}[Proof of Theorem \ref{thm:mavlyutov}(A)]
By Lemma~\ref{lemma:sigma_tilde_is_full_dimensional} and Lemma~\ref{lemma:cox_equations_slicing_cone} it is enough to show that the toric morphism $\toric{\CC}{\sigma} \to \toric{\CC}{\tilde{\sigma}}$ is a closed embedding. 

Before proving this, we shall prove the following claim:
\begin{equation} \label{eq:equality_minima_floor}
\forall u \in \sigma^\vee \cap M, \qquad \sum_{i=0}^k \left\lfloor \min_{Q_i} u \right\rfloor = \left\lfloor \min_Q u \right\rfloor.
\end{equation}
Firstly we show that the minimum of $u$ on $Q$ is attained on a vertex of $Q$; this comes from the strong convexity of $\sigma$ as follows.
By (i) $\rec(Q)$ is contained in $\sigma$ and so is a strongly convex cone. We have
\begin{equation} \label{eq:Q_vertQ_recQ}
Q = \conv{\vertices(Q)} + \rec(Q).
\end{equation}
Since $u \in \sigma^\vee$, $u$ is non-negative on $\rec(Q)$. Therefore there exists a vertex $v$ of $Q$ such that $\min_Q u = \langle u, v \rangle$. Now we prove the claim \eqref{eq:equality_minima_floor}. By (iv') we may find vertices $v_i \in \vertices(Q_i)$, $i=0,1,\dots, k$, such that $v = v_0 + v_1 + \cdots + v_k$ and they are all integral with at most one exception. This implies that the numbers $\langle u, v_0 \rangle$, $\langle u, v_1 \rangle$, \dots, $\langle u, v_k \rangle$ are all integral with at most one exception. Therefore
\begin{equation*}
\sum_{i=0}^k \left\lfloor \langle u , v_i \rangle \right\rfloor = \left\lfloor \langle u, v \rangle \right\rfloor.
\end{equation*}
But $\min_Q u = \langle u, v \rangle$ and it is clear that $\min_{Q_i} u = \langle u, v_i \rangle$ for $i = 0,1, \dots, k$. Therefore we have proved \eqref{eq:equality_minima_floor}.

Now we prove that the toric morphism $\toric{\CC}{\sigma} \to \toric{\CC}{\tilde{\sigma}}$ is a closed embedding. Equivalently, we have to show that the semigroup homomorphism
\[
\phi \colon \tilde{\sigma}^\vee \cap \tilde{M} \to \sigma^\vee \cap M
\]
is surjective. Let $u \in \sigma^\vee \cap M$ and consider
\begin{equation*}
\tilde{u} = u - \sum_{i=1}^k \left\lfloor \min_{Q_i} u \right\rfloor e_i^* \in \tilde{M};
\end{equation*}
if we prove that $\tilde{u} \in \tilde{\sigma}^\vee$ we have finished because the equality $\phi(\tilde{u}) = u$ obviously holds true. It is clear that $\tilde{u}$ is non-negative on $\sigma$ and it is very easy to show that $\tilde{u}$ is non-negative on $Q_i +e_i$, for each $i=1, \dots, k$. So it remains to show that $\tilde{u}$ is non-negative on $Q_0 - e_1 - \dots - e_k$. If $q \in Q_0$, then
\begin{align*}
\langle \tilde{u}, q - e_1 - \cdots - e_k \rangle = \langle u, q \rangle + \sum_{i=1}^k \left\lfloor \min_{Q_i} u \right\rfloor \geq \\
\geq \left\lfloor \min_{Q_0} u \right\rfloor + \sum_{i=1}^k \left\lfloor \min_{Q_i} u \right\rfloor
= \left\lfloor \min_{Q} u \right\rfloor \geq 0,
\end{align*}
where the last equality is \eqref{eq:equality_minima_floor} and the last inequality holds because of (i).

This concludes the proof of Theorem \ref{thm:mavlyutov}(A).
\end{proof}

\begin{lemma} \label{lemma:complete_intersection_flatness_neighbourhood}
Let $(A, \frakm, \kappa)$ be an artinian local ring and $B$ be a flat $A$-algebra of finite type. Let $b_1, \dots, b_k \in B$ generate the ideal $J$ of $B$. 
If $b_1, \dots, b_k$ is a $(B \otimes_A \kappa)$-regular sequence, then $B/J$ is flat over $A$.
\end{lemma}

\begin{proof}
Let $P$ be a prime ideal of $B$. Since $\frakm$ is the unique prime ideal of $A$, we have $\frakm = P \cap A$ and $A \to B_P$ is a local homomorphism. We need to show that $(B/J)_P = B_P / J B_P$ is flat over $A$. If $J \nsubseteq P$, then $(B/J)_P = 0$ and we are done. If $J \subseteq P$, then we conclude by \cite[Corollary to Theorem 22.5]{matsumura}.
\end{proof}

\begin{proof}[{Proof of Theorem~\ref{thm:mavlyutov}(B)}]
From \eqref{eq:Q_vertQ_recQ} and (v), we have that $w$ is non-negative on $\rec(Q)$ and $\min_Q w = \langle w, v \rangle$ for some vertex $v$ of $Q$. By (iv') we may find vertices $v_i \in \vertices(Q_i)$, $i=0,1,\dots, k$, such that $v = v_0 + v_1 + \cdots + v_k$ and they are all integral with at most one exception. This implies that the numbers $\langle w, v_0 \rangle$, $\langle w, v_1 \rangle$, \dots, $\langle w, v_k \rangle$ are all integral with at most one exception. Therefore
\begin{equation*}
\sum_{i=0}^k \left\lfloor \langle w , v_i \rangle \right\rfloor = \left\lfloor \langle w, v \rangle \right\rfloor.
\end{equation*}
But $\min_Q w = \langle w, v \rangle$ and it is clear that $\min_{Q_i} w = \langle w, v_i \rangle$ for $i = 0,1, \dots, k$. Therefore we have proved the equality
\begin{equation} \label{eq:equality_minima_floor_h}
\sum_{i=0}^k \left\lfloor \min_{Q_i} w \right\rfloor = \left\lfloor \min_Q w \right\rfloor.
\end{equation}

Now we show that the trinomials \eqref{eq:trinomials} are elements of the polynomial ring
\[
\CC[t_1, \dots, t_k][x_\xi \mid \xi \in \tilde{\sigma}(1)],
\]
 which is the homogeneous coordinate ring of $\toric{\CC[t_1, \dots, t_k]}{\tilde{\sigma}}$. It is enough to show that every Cox coordinate appearing in the third monomial in \eqref{eq:trinomials} has a non-negative exponent. Fix a ray $\xi$ of $\tilde{\sigma}$. We may distinguish three cases as follows.
\begin{itemize}

\item $\xi$ passes through a vertex of $Q_0 - e_1 - \dots - e_k$. Then $\xi = \lambda (v_0 - e_1 - \cdots - e_k)$, for some $\lambda \in \NN^+$ and $v_0 \in \vertices(Q_0)$. Then
\begin{gather*}
\langle \tilde{w}, \xi \rangle = \lambda \langle w, v_0 \rangle + \lambda \sum_{i=1}^k \left\lfloor \min_{Q_i} w \right\rfloor 
\geq \lambda \sum_{i=0}^k \left\lfloor \min_{Q_i} w \right\rfloor 
= \lambda \left\lfloor \min_{Q} w \right\rfloor 
\geq - \lambda,
\end{gather*}
where the last equality holds by \eqref{eq:equality_minima_floor_h} and the last inequality holds by (v). Therefore the exponent of $x_\xi$ in the third trinomial in \eqref{eq:trinomials}, which is $\langle \tilde{w}, \xi \rangle + \lambda$, is non-negative.
\item $\xi$ passes through a vertex of $Q_i + e_i$, for some $1 \leq i \leq k$. Then $\xi = \lambda (v + e_i)$, for some $\lambda \in \NN^+$ and $v \in \vertices(Q_i)$. Then $\langle \tilde{w}, \xi \rangle = \lambda \langle w, v \rangle - \lambda \left\lfloor \min_{Q_i} w \right\rfloor \geq 0$.
\item $\xi$ is a ray of $\sigma$ too. We need to show that $\langle \tilde{w}, \xi \rangle = \langle w, \xi \rangle$ is non-negative. For a contradiction assume that $\langle w, \xi \rangle < 0$. Therefore a positive multiple of $\xi$ lies in the polyhedron $P := \sigma \cap \{ n \in N_\RR \mid \langle w, n \rangle = -1 \}$. Since $\rec(P)$ is strongly convex, $P = \conv{\vertices(P)} + \rec(P)$. By (vi) we obtain that $\xi = \lambda q + r$, for some $\lambda >0$, $q \in Q$, $r \in \rec(P)$. Since $\lambda q$ and $r$ are both in $\sigma$ and $\xi$ is a ray of $\sigma$, we have that
$\xi = \mu q$ for some $\mu \geq 0$.
 From (iii) we have that $\xi$ is in $\cone{Q_0-e_1-\cdots-e_k, Q_1+e_1, \dots, Q_k+e_k}$. This contradicts the fact that $\xi$ is a ray of both $\sigma$ and $\tilde{\sigma}$.
\end{itemize}

Now the closed subscheme $\cX$ is well defined.
We need to show that the restriction of $\cX \to \AA^k_\CC$ to any infinitesimal neighbourhood of $O \in \AA^k_\CC$ is flat.

Fix a $(t_1, \dots, t_k)$-primary ideal $\frakq$. Consider the local artinian $\CC$-algebra $A = \CC[t_1, \dots, t_k] / \frakq$. We need to show that $\cX \times_{\AA^k_\CC} \Spec A \to \Spec A$ is flat. The homogeneous coordinate ring of $\toric{A}{\tilde{\sigma}}$ is the polynomial $A$-algebra $B = A[ x_\xi \mid \xi \in \tilde{\sigma}(1) ]$. By (A) the trinomials \eqref{eq:trinomials} form a $(B \otimes_A \CC)$-regular sequence. By Lemma~\ref{lemma:complete_intersection_flatness_neighbourhood} 
the homogeneous ideal $J \subseteq B$ generated by the trinomials \eqref{eq:trinomials} is such that $B/J$ is flat over $A$. By Lemma~\ref{lemma:flatness_sheafification_graded_module} the sheafification of the $G_{\tilde{\sigma}}$-graded $B$-module $B/J$ is a coherent sheaf on $\toric{A}{\tilde{\sigma}}$ which is flat over $\Spec A$. This sheaf is the structure sheaf of the closed subscheme $\cX \times_{\AA^k_\CC} \Spec A$ of $\toric{A}{\tilde{\sigma}}$. Therefore we have proved that $\cX \times_{\AA^k_\CC} \Spec A$ is flat over $\Spec A$.

This concludes the proof of Theorem~\ref{thm:mavlyutov}(B).
\end{proof}

\section{Deformations of affine toric pairs} \label{sec:deformations of toric affine pairs}

If in Theorem~\ref{thm:mavlyutov} we assume that $(Q,Q_0,Q_1, \dots,Q_k,w)$ is a $\partial$-deformation datum, then Mavlyutov's construction of deformations of affine toric varieties, which appears in \cite{mavlyutov} and is rewritten in \S\ref{sec:deformations_affine_mavlyutov}, actually gives deformations of their toric boundary too. 

More precisely, in the setting of Theorem~\ref{thm:mavlyutov} with the additional hypothesis (iv), we construct a reduced divisor $\cD$ in the toric variety $\tilde{X} = \toric{\CC}{\tilde{\sigma}}$ such that $\cD \cap X$ is the toric boundary $\partial X$ of $X$.  Theorem~\ref{thm:mavlyutov} constructs a formal deformation $\cX \to \AA^k_\CC$ of $X$ as a closed subscheme in the trivial family $\tilde{X} \times_{ \CC} \AA^k_\CC$; then one can see that the closed subscheme $\cX \cap (\cD \times_{\CC} \AA^k_\CC)$ gives a deformation of $\partial X$. In other words, $\cX \cap (\cD \times_\CC \AA^k_\CC) \into \cX \to \AA^k_\CC$ induces a formal deformation of the toric pair $(X,  \partial X)$. This is the content of the following theorem, which is the precise formulation of Theorem~\ref{thm:deformations_affine_toric_pairs_easy}.

\begin{theorem} \label{thm:deformations_affine_toric_pairs}
Let $N$ be a lattice, let $\sigma \subseteq N_\RR$ be a strongly convex rational polyhedral cone with dimension $\rank N$, let $(Q,Q_0,Q_1, \dots, Q_k,w)$ be a $\partial$-deformation datum for $(N,\sigma)$, and let $\tilde{N}$, $\tilde{\sigma}$ and $\tilde{w}$ be as in Notation~\ref{not:tilde_N_tilde_sigma_tilde_w}.

Let $X$ be the affine toric variety associated to $\sigma$, let $\partial X$ be the toric boundary of $X$, and let $\tilde{X}$ be the affine toric variety associated to $\tilde{\sigma}$.
Consider the reduced effective divisor $\cD$ on $\tilde{X}$ defined by the homogeneous ideal generated by the following monomial in the Cox coordinates of $\tilde{X}$:
\begin{equation} \label{eq:monomial_theorem_deformations_affine_pairs}
\prod_{\substack{\xi \in \tilde{\sigma}(1) \colon \\ \forall i \in \{ 1, \dots, k \}, \langle e_i^*, \xi \rangle \leq 0}} x_\xi.
\end{equation}

Let $t_1, \dots, t_k$ be the standard coordinates on $\AA^k_\CC$. Consider the closed subscheme $\cX$ of $\tilde{X} \times_{\Spec \CC} \AA^k_\CC = \toric{\CC[t_1, \dots, t_k]}{\tilde{\sigma}}$ defined by the homogeneous ideal generated by the following trinomials in Cox coordinates:
\begin{equation} \label{eq:trinomials_theorem_deformations_affine_pairs}
\prod_{\substack{\xi \in \tilde{\sigma}(1) \colon \\ \langle e_i^*, \xi \rangle > 0}} x_\xi^{\langle e_i^*, \xi \rangle} - 
\prod_{\substack{\xi \in \tilde{\sigma}(1) \colon \\ \langle e_i^*, \xi \rangle < 0}} x_\xi^{-\langle e_i^*, \xi \rangle} -
t_i \prod_{\xi \in \tilde{\sigma}(1)} x_\xi^{\langle \tilde{w}, \xi \rangle} \prod_{\substack{\xi \in \tilde{\sigma}(1) \colon \\ \langle e_i^*, \xi \rangle < 0}} x_\xi^{-\langle e_i^*, \xi \rangle}
\end{equation}
for $i=1, \dots, k$.

Then the diagram
\[
\xymatrix{
\cX \cap (\cD \times_{\Spec \CC} \AA^k_\CC) \ar[rd] \ar@{^{(}->}[r] & \cX \ar[d] \\
& \AA^k_\CC
}
\]
induces a formal deformation of the toric pair $(X, \partial X)$ over $\CC [ \![ t_1, \dots, t_k ] \! ]$.
\end{theorem}

\begin{example} \label{ex:cA1} 
	In the lattice $N = \ZZ^3$ consider the cone
	\[
	\sigma = \cone{ \begin{pmatrix}
		1 \\ 1 \\ 0
		\end{pmatrix},
		\begin{pmatrix}
		-1 \\ 1 \\ 0
		\end{pmatrix},
		\begin{pmatrix}
		0 \\ 0 \\ 1
		\end{pmatrix}  } \subseteq N_\RR.
	\]
	The corresponding affine toric variety $X = \toric{\CC}{\sigma}$ is the $\mathrm{cA}_1$-singularity:
	\[
	X = \Spec \CC[x,y,z,u]/(xy-u^2)
	\]
	where the variables $x,y,z,u$ correspond to the following generators of the monoid $\sigma^\vee \cap M$: $(1,1,0)$, $(-1,1,0)$, $(0,0,1)$, $(0,1,0)$. The toric boundary is
	\[
	\partial X = \Spec \CC[x,y,z,u]/(xy-u^2, zu).
	\]
	
	Fix $p \in \NN$ and consider the $\partial$-deformation datum $(Q_0+Q_1, Q_0, Q_1,w)$ for $(N,\sigma)$ (see Definition~\ref{def:deformation_datum}) given by the polytopes 
	\begin{align*}
	Q_0 = \left\{ \begin{pmatrix}
	-1/2 \\ 1/2 \\ 0
	\end{pmatrix} \right\}
	\qquad \text{and} \qquad
	Q_1 = \conv{\begin{pmatrix}
		0 \\ 0 \\ 0
		\end{pmatrix}
		,
		\begin{pmatrix}
		1 \\ 0 \\ 0
		\end{pmatrix}  },
	\end{align*}
	and by the character $w = (0,-2,p) \in M$.
	Following Notation~\ref{not:tilde_N_tilde_sigma_tilde_w}, inside the bigger lattice $\tilde{N} = \ZZ^4$ we construct the cone
\begin{align*}
\tilde{\sigma} &= \cone{
\begin{pmatrix}
1 \\ 1 \\ 0 \\ 0
\end{pmatrix},
\begin{pmatrix}
-1 \\ 1 \\ 0 \\ 0
\end{pmatrix},\begin{pmatrix}
0 \\ 0 \\ 1 \\ 0
\end{pmatrix},\begin{pmatrix}
-1/2 \\ 1/2 \\ 0 \\ -1
\end{pmatrix},\begin{pmatrix}
0 \\ 0 \\ 0 \\ 1
\end{pmatrix},\begin{pmatrix}
1 \\ 0 \\ 0 \\ 1
\end{pmatrix}
} \\
&= \cone{
\begin{pmatrix}
0 \\ 0 \\ 1 \\ 0
\end{pmatrix},\begin{pmatrix}
-1 \\ 1 \\ 0 \\ -2
\end{pmatrix},\begin{pmatrix}
0 \\ 0 \\ 0 \\ 1
\end{pmatrix},\begin{pmatrix}
1 \\ 0 \\ 0 \\ 1
\end{pmatrix}
}.
\end{align*}	
One can see that these last four vectors are a basis of $\tilde{N} = \ZZ^4$, therefore the affine toric variety $\tilde{X}$ associated to $\tilde{\sigma}$ is the affine space $\AA^4_\CC = \Spec \CC[x,y,z,u]$.
Now the variables $x,y,z,u$ correspond to the following elements of $\tilde{M} = \ZZ^4$: $(1,1,0,0)$, $(-1,1,0,1)$, $(0,0,1,0)$, $(0,1,0,0)$, which form the dual basis to the primitive generators of the rays of $\tilde{\sigma}$. In this particular case the Cox coordinates of $\tilde{X}$ coincide with $x,y,z,u$.

The trinomial \eqref{eq:trinomials_theorem_deformations_affine_pairs} is $xy-u^2 -t z^p$ and the monomial 
\eqref{eq:monomial_theorem_deformations_affine_pairs} is $zu$.
 By Theorem~\ref{thm:deformations_affine_toric_pairs} we consider the following closed subschemes of $\tilde{X} \times_\CC \Spec \CC[t]$:
	\begin{align*}
	\cX &= \Spec \CC[x,y,z,u,t]/(xy - u^2 - t z^p), \\
	\cB &= \Spec \CC[x,y,z,u,t]/(xy - u^2 - t z^p, zu).
	\end{align*}
	The diagram
	\[
	\xymatrix{
		\cB \ \ar@{^(->}[r] \ar[rd] & \cX \ar[d] \\
		& \Spec \CC[t]
	}
	\]
	induces a formal deformation of the pair $(X, \partial X)$ over $\CC [ \! [ t ] \! ]$, by base changing to $\Spec \CC[t]/(t^{n+1})$ for each $n \in \NN$.
\end{example}

The rest of this section is devoted to the proof of Theorem~\ref{thm:deformations_affine_toric_pairs}.

\begin{lemma} \label{lemma:monomials_free_module}
Let $S$ be a polynomial ring over $\CC$ in finitely many indeterminates. Let $m_1, \dots, m_t \in S \setminus \{ 1 \}$ be some monomials such that the $t$ sets of indeterminates appearing in these monomials have empty pairwise intersections. Let $R$ be the $\CC$-subalgebra of $S$ generated by $m_1, \dots, m_t$. Then $m_1, \dots, m_t$ are algebraically independent over $\CC$ and $S$ is a free $R$-module.
\end{lemma}

\begin{proof}
It is clear that the monomials $m_1, \dots, m_t$ are algebraically independent over $\CC$. Another way to see this is to notice that they form a regular sequence in $S$ and then use \cite[Exercise~16.6]{matsumura}.

Now we want to prove that $S$ is a free $R$-module. For each $i=1, \dots, t$, let $S_i$ be the polynomial ring over $\CC$ in the indeterminates that appear in $m_i$ and let $R_i \subseteq S_i$ be the $\CC$-subalgebra generated by $m_i$.
Let $S_0$ be the polynomial ring over $\CC$ in the indeterminates of $S$ that do not appear in any $m_i$'s. If we prove that $S_i$ is a free $R_i$-module for each $i=1, \dots, t$, then $S = S_0 \otimes_\CC S_1 \otimes_\CC \cdots \otimes_\CC S_t$ will be free over $R = R_1 \otimes_\CC \cdots \otimes_\CC R_t$.

Therefore we may assume that $t = 1$ and that all the indeterminates of $S$ appear in $m := m_1$, i.e.\ $S = \CC[x_1, \dots, x_n]$ and $m = x_1^{a_1}\cdots x_n^{a_n}$ with $a_1, \dots, a_n \in \NN^+$. The set
\[
\bigcup_{i=1}^n \{ x_1^{b_1} \cdots x_n^{b_n} \in S \mid b_i < a_i \}
\]
is a free basis of $S$ as $R$-module.
\end{proof}

\begin{lemma} \label{lemma:regular_sequence_binomials_monomial}
Let $S$ be a polynomial ring over $\CC$ in finitely many indeterminates. Let $y_1, \dots, y_k, z_1, \dots, z_r\in S \setminus \{ 1 \}$ be some monomials such that the $k+r$ sets of indeterminates appearing in these monomials have empty pairwise intersections. Fix $c_1, \dots, c_r \in \NN$ and consider the monomial $z_0 = z_1^{c_1} \cdots z_r^{c_r}$. Then
\begin{equation} \label{eq:lemma_regular_sequence_binomials_monomial}
y_1 - z_0, \dots, y_k - z_0, z_1 \cdots z_r
\end{equation}
is a regular sequence.
\end{lemma}

\begin{proof}
Consider the polynomial ring $R = \CC[Y_1, \dots, Y_k, Z_1, \dots, Z_r]$ where the $Y_i$'s and the $Z_j$'s are indeterminates. Consider the $\CC$-algebra homomorphism $\phiv \colon R \to S$ defined by $Y_i \mapsto y_i$ and $Z_j \mapsto z_j$. By Lemma~\ref{lemma:monomials_free_module}, $\phiv$ is injective and flat.

Consider the monomial $Z_0 = Z_1^{c_1} \cdots Z_r^{c_r}$. Consider the $\CC$-algebra automorphism $\theta$ of $R$ that fixes the $Z_j$'s and maps $Y_i$ to $Y_i - Z_0$. By applying $\theta$ to the regular sequence $Y_1, \dots, Y_k, Z_1\cdots Z_r$ we get the regular sequence
\begin{equation*}
Y_1 - Z_0, \dots, Y_k - Z_0, Z_1 \cdots Z_r.
\end{equation*}
Now, by applying the flat map $\phiv$ to this regular sequence, we get that the sequence \eqref{eq:lemma_regular_sequence_binomials_monomial} is regular.
\end{proof}

\begin{remark} \label{rmk:fischer_shapiro_new_proof}
A slight generalisation of Lemma~\ref{lemma:regular_sequence_binomials_monomial} can be used to prove that the binomials in Theorem~\ref{thm:mavlyutov}(A) form a regular sequence, without using \cite{fischer_shapiro} and Lemma~\ref{lemma:fischer_shapiro}.
\end{remark}

\begin{proof}[Proof of Theorem~\ref{thm:deformations_affine_toric_pairs}]
By Theorem~\ref{thm:mavlyutov}, it is enough to deal with the toric boundary. Here we adopt some notations used in the proof of Lemma~\ref{lemma:cox_equations_slicing_cone}. Let $I$ be the kernel of the surjective ring homomorphism $\psi \colon \CC[\tilde{\sigma}^\vee \cap \tilde{M}] \to \CC[\sigma^\vee \cap M]$ that is associated to the surjective semigroup homomorphism $\phi \colon \tilde{\sigma}^\vee \cap \tilde{M} \to \sigma^\vee \cap M$ given by $u + a_1 e_1^* + \cdots + a_k e_k^* \mapsto u$. The ideal of the toric boundary $\partial X$ in $X$ is
\begin{equation*}
\bigoplus_{u \in \interior(\sigma^\vee) \cap M} \CC \chi^u.
\end{equation*}
Therefore the ideal of $\partial X$ in $\tilde{X}$ is
\begin{align*}
\overline{I} &:= \psi \inv \left( \bigoplus_{u \in \interior(\sigma^\vee) \cap M} \CC \chi^u \right) \\ &= I + \sum_{u + a_1 e_1^* + \cdots + a_k e_k^* \in  ((\interior(\sigma^\vee) \cap M) \times \ZZ^k) \cap \tilde{\sigma}^\vee}  \CC \chi^{u + a_1 e_1^* + \cdots + a_k e_k^*}.
\end{align*}

Now we consider the Cox ring of $\tilde{X}$: $S = \CC[x_\xi \mid \xi \in \tilde{\sigma}(1)]$ with its $G_{\tilde{\sigma}}$-grading. In the proof of Theorem~\ref{thm:mavlyutov}(A) we had the following description of the rays of $\tilde{\sigma}$.
\begin{itemize}
\item Rays passing through the vertices of $Q_0 - e_1 - \cdots -e_k$. We denote by $z_{0,1}, \dots, z_{0,s_0}$ the corresponding Cox coordinates.
\item Rays passing through the vertices of $Q_i + e_i$, as $ i = 1, \dots, k$. We denote by $y_{i,1}, \dots, y_{i, s_i}$ the corresponding Cox coordinates.
\item Rays of $\sigma$ that are not in the cone generated by the previous rays. We denote by $z_{\sigma, 1}, \dots, z_{\sigma, s_\sigma}$ the corresponding Cox coordinates.
\end{itemize}
Consider the following monomials in the Cox coordinates of $\tilde{X}$:
\begin{align*}
y_i &= \prod_{\substack{\xi \in \tilde{\sigma}(1) \colon \\ \langle e_i^*, \xi \rangle > 0}} x_\xi^{\langle e_i^*, \xi \rangle} = y_{i,1} \cdots y_{i, s_i} & & \qquad \text{for each } i \in \{ 1, \dots, k \}, \\
z_0 &= \prod_{\substack{\xi \in \tilde{\sigma}(1) \colon \\ \langle e_i^*, \xi \rangle < 0}} x_\xi^{-\langle e_i^*, \xi \rangle} = z_{0,1}^{c_1} \cdots z_{0, s_0}^{c_{s_0}} & & \qquad \text{for any } i \in \{ 1, \dots, k \}, \\
z_0^\mathrm{red} &= \prod_{\substack{\xi \in \tilde{\sigma}(1) \colon \\ \langle e_i^*, \xi \rangle < 0}} x_\xi = z_{0,1} \cdots z_{0, s_0} & & \qquad \text{for any } i \in \{ 1, \dots, k \}, \\
z_\sigma &= \prod_{\substack{\xi \in \tilde{\sigma}(1) \colon \\ \langle e_i^*, \xi \rangle = 0}} x_\xi = z_{\sigma, 1} \cdots z_{\sigma, s_\sigma} & & \qquad \text{for any } i \in \{ 1, \dots, k \}, \\
z &= \prod_{\substack{\xi \in \tilde{\sigma}(1) \colon \\ \langle e_i^*, \xi \rangle \leq 0}} x_\xi = z_0^\mathrm{red} z_\sigma  & & \qquad \text{for any } i \in \{ 1, \dots, k \}.
\end{align*}
The exponents $c_0, \dots, c_{s_0}$ are the minimal positive integers by which we have to multiply the vertices of $Q_0$ to get lattice points.
Here we have used (iv) in Definition~\ref{def:deformation_datum} to deduce that $y_i$ are reduced monomials.
We see that $y_i$ are exactly the ones used in the proof of Lemma~\ref{lemma:cox_equations_slicing_cone}, whereas the monomials $z_1, \dots, z_k$ there coincides with $z_0$ in our case. We see that $y_i - z_0$ is the binomial obtained from the trinomial \eqref{eq:trinomials_theorem_deformations_affine_pairs} by setting $t_i=0$ and $z$ is the monomial in \eqref{eq:monomial_theorem_deformations_affine_pairs}.
Let $J \subseteq S$ be the ideal generated by $y_1 - z_0, \dots, y_k - z_0$ and let $\overline{J} = J + S z$. We already know, from Lemma~\ref{lemma:cox_equations_slicing_cone} or Theorem~\ref{thm:mavlyutov}, that the Cox isomorphism between $\CC[\tilde{\sigma}^\vee \cap M]$ and $S_0 \subseteq S$ maps
 the ideal $I$ onto the degree zero part of the ideal $J$, i.e.\ $\mathrm{Cox}(I) = J \cap S_0$. We have to prove that
\begin{equation} \label{eq:equality_cox_ideals_toric_boundary}
\mathrm{Cox}(\overline{I}) = \overline{J} \cap S_0.
\end{equation}
This equality will imply that the scheme-theoretic intersection $X \cap \cD$ coincides with $\partial X$.

We now prove the containment $\subseteq$ in \eqref{eq:equality_cox_ideals_toric_boundary}.  Since $\mathrm{Cox}(I) \subseteq J  \subseteq \overline{J} $, it is enough to show that $\mathrm{Cox}(\chi^{\tilde{u}}) = x^{\tilde{u}} \in \overline{J}$ for all $\tilde{u} = u + a_1 e_1^* + \cdots + a_k e_k^* \in \tilde{\sigma}^\vee \cap \tilde{M}$  such that $u \in \interior(\sigma^\vee)$. We have that $z_\sigma$ divides $x^{\tilde{u}}$ because $u$ is in the strict interior of $\sigma^\vee$. Since $\tilde{u} \in \tilde{\sigma}^\vee$, $\tilde{u}$ cannot take negative values on $Q_0 - e_1 - \cdots - e_k$, $Q_1 + e_1, \dots, Q_k + e_k$. If $\tilde{u}$ is strictly positive on $Q_0 - e_1 - \cdots - e_k$, then $z_0^\mathrm{red}$ divides $x^{\tilde{u}}$, and hence $z = z_0^\mathrm{red} z_\sigma$ divides $x^{\tilde{u}}$, which implies that $x^{\tilde{u}}$ lies in $\overline{J}$ and we are done. So we may assume that $0 = \min_{Q_0 - e_1 - \cdots - e_k} \tilde{u} = \min_{Q_0} u - a_1 - \cdots - a_k$. Therefore, since $u \in \interior(\sigma^\vee)$ and $0 \notin Q$, we have
\begin{align*}
0 < \min_Q u 
= \min_{Q_0} u + \min_{Q_1} u + \cdots + \min_{Q_k} u 
= \sum_{i=1}^k \left( a_i + \min_{Q_i} u \right) 
= \sum_{i=1}^k \min_{Q_i + e_i} \tilde{u}.
\end{align*}
So, there exists $i \in \{ 1, \dots, k \}$ such that $\min_{Q_i + e_i} \tilde{u} > 0$. This implies that $y_i$ divides $x^{\tilde{u}}$, i.e.\ there exists a monomial $p$ such that $x^{\tilde{u}} = p y_i$. Since $z_\sigma \vert x^{\tilde{u}}$, we know that $z_\sigma \vert p$. By writing $x^{\tilde{u}} = p (y_i - z_0) + p z_0$  and by noting that $z$ divides $p z_0$, we conclude that $x^{\tilde{u}}$ lies in $\overline{J}$.

We now prove the containment $\supseteq$ in \eqref{eq:equality_cox_ideals_toric_boundary}. By using the same argument as in the second part of the proof of Lemma~\ref{lemma:cox_equations_slicing_cone}, it is enough to show that if $\tilde{u} = u + a_1 e_1^* + \cdots + a_k e_k^* \in \tilde{\sigma}^\vee \cap \tilde{M}$ is such that $x^{\tilde{u}} = p z$ for some monomial $p \in S$ then $u \in \interior(\sigma^\vee)$. Since $z$ divides $x^{\tilde{u}}$, we see that $\tilde{u}$ is strictly positive on $Q_0 - e_1 - \cdots - e_k$ and on the rays of $\sigma$ that are not in the cone generated by $Q_0 - e_1 - \cdots - e_k, Q_1 + e_1, \dots, Q_k + e_k$. Now we want to prove that $u$ is strictly positive on the non-zero elements of $\sigma$; if $v \in \sigma$  we can write $v = \lambda(q_0 - e_1 - \cdots - e_k) + \lambda(q_1 + e_1) + \cdots + \lambda (q_k + e_k) + v_\sigma = \lambda (q_0 + q_1 + \cdots + q_k) + v_\sigma$, for some $\lambda  \geq 0$, $q_i \in Q_i$, and $v_\sigma$ in the cone generated by the rays of $\sigma$ that are not in the cone generated by $Q_0 - e_1 - \cdots - e_k, Q_1 + e_1, \dots, Q_k + e_k$. We have
\begin{equation*}
\langle u, v \rangle = \lambda \left[ \langle \tilde{u}, q_0 - e_1 - \cdots - e_k \rangle + \langle \tilde{u}, q_1 + e_1 \rangle + \cdots + \langle \tilde{u}, q_k + e_k \rangle \right] + \langle \tilde{u}, v_\sigma \rangle.
\end{equation*}
Since $v \neq 0$, we have that either $\lambda > 0$ or $v_\sigma \neq 0$; this implies $\langle u, v \rangle > 0$.

This concludes the proof of the equality \eqref{eq:equality_cox_ideals_toric_boundary} and, consequently, of the fact that $X \cap \cD = \partial X$.

By Lemma~\ref{lemma:regular_sequence_binomials_monomial} we have that $y_1 - z_0, \dots, y_k - z_0, z$ is a regular sequence. Adapting the proof of Theorem~\ref{thm:mavlyutov}(B) we conclude.
\end{proof}

\section{Deformations of projective toric varieties}
\label{sec:deformations_projective_toric}

In this section we study deformations of polarised projective toric varieties. Our strategy is to deform the corresponding affine cones thanks to Mavlyutov's theorem (Theorem~\ref{thm:mavlyutov}) and then apply the Proj functor. We will use the lemmata in \S\ref{sec:polarised_projective_toric_varieties}.

\begin{theorem}\label{thm:deformations_projective_toric_varieties}
Let $N$ be a lattice of rank $n$, let $X$ be a projective $T_N$-toric variety, and let $D$ be an ample torus-invariant $\QQ$-Cartier $\QQ$-divisor on $X$. Let $\tau$ be the $(n+1)$-dimensional cone in the lattice $N_0 = N \oplus \ZZ e_0$ associated to the pair $(X,D)$ as in Lemma~\ref{lemma:polarised_projective_varieties}. Let $(Q,Q_0,Q_1, \dots, Q_k,w)$ be a deformation datum for $(N_0, \tau)$ with $w \in M \subseteq M_0$. Consider the lattices $\tilde{N} = N \oplus \ZZ e_1 \oplus \cdots \oplus \ZZ e_k$ and $\tilde{N}_0 = N \oplus \ZZ e_0 \oplus \ZZ e_1 \oplus \cdots \oplus \ZZ e_k$. Let $\tilde{\tau} \subseteq (\tilde{N}_0)_\RR$ and $\tilde{w} \in \tilde{M} \subseteq \tilde{M}_0$ be as in Notation~\ref{not:tilde_N_tilde_sigma_tilde_w}. Let $(\tilde{X}, \tilde{D})$ be the polarised projective toric variety associated to the cone $\tilde{\tau}$ via Lemma \ref{lemma:polarised_projective_varieties}.

\smallskip

{\rm (A)}  Then the inclusion $\tau \into \tilde{\tau}$ induces a toric closed embedding $X \into \tilde{X}$ which identifies $X$ with the closed subscheme of $\tilde{X}$ associated to the homogeneous ideal generated by the following binomials in the Cox coordinates of $\tilde{X}$:
\begin{equation} \label{eq:binomials_theorem_deformations_projective_varieties}
\prod_{\substack{\rho \in \tilde{\Sigma}(1) \colon \\ \langle e_i^*, \rho \rangle > 0}} x_\rho^{\langle e_i^*, \rho \rangle} - 
\prod_{\substack{\rho \in \tilde{\Sigma}(1) \colon \\ \langle e_i^*, \rho \rangle < 0}} x_\rho^{-\langle e_i^*, \rho \rangle}
\end{equation}
for $i = 1, \dots, k$, where $\tilde{\Sigma}$ is the fan of $\tilde{X}$ in $\tilde{N}$. Moreover, the $k$ binomials in \eqref{eq:binomials_theorem_deformations_projective_varieties} form a regular sequence.

\smallskip

{\rm (B)} 
Let $t_1, \dots, t_k$ be the standard coordinates on $\AA^k_\CC$. Consider the closed subscheme $\cX$ of $\tilde{X} \times_{\Spec \CC} \AA^k_\CC = \toric{\CC[t_1, \dots, t_k]}{\tilde{\Sigma}}$ defined by the homogeneous ideal generated by the following trinomials in Cox coordinates:
\begin{equation} \label{eq:trinomials_theorem_deformations_projective_varieties}
\prod_{\substack{\rho \in \tilde{\Sigma}(1) \colon \\ \langle e_i^*, \rho \rangle > 0}} x_\rho^{\langle e_i^*, \rho \rangle} - 
\prod_{\substack{\rho \in \tilde{\Sigma}(1) \colon \\ \langle e_i^*, \rho \rangle < 0}} x_\rho^{-\langle e_i^*, \rho \rangle} -
t_i \prod_{\rho \in \tilde{\Sigma}(1)} x_\rho^{\langle \tilde{w}, \rho \rangle} \prod_{\substack{\rho \in \tilde{\Sigma}(1) \colon \\ \langle e_i^*, \rho \rangle < 0}} x_\rho^{-\langle e_i^*, \rho \rangle}
\end{equation}
for $i=1, \dots, k$. Then the morphism $\cX \to \AA^k_\CC$ induces a deformation of $X$ over $\CC [ \![ t_1, \dots, t_k ] \! ]$ and over an open neighbourhood of the origin in $\AA^k_\CC$.
\end{theorem}

The rest of this section is devoted to the proof of Theorem \ref{thm:deformations_projective_toric_varieties}. 

\begin{proof}[Proof of Theorem \ref{thm:deformations_projective_toric_varieties}(A)]
By Lemma~\ref{lemma:sigma_tilde_is_full_dimensional} $\tilde{\tau}$ is a $(n+1+k)$-dimensional strongly convex rational polyhedral cone in $\tilde{N}_0$. It is clear that $e_0 \in \tilde{\tau}$.

Now we show that $e_0$ is in the interior of $\tilde{\tau}$: it is enough to show that, if $\tilde{u} = u + \sum_{i=0}^k h_i e_i^* \in \tilde{\tau}^\vee \cap \tilde{M}_0$ and $h_0 = 0$, then $\tilde{u} = 0$. Since $\tau \subseteq \tilde{\tau}$, we have that $u$ is non-negative on $\tau$; but $e_0$ is in the interior of $\tau$, so $u=0$. By evaluating $\tilde{u} = \sum_{i=1}^k h_i e_i^*$ on $Q_0 -e_1-\cdots-e_k$, $Q_1 + e_1$, \dots, $Q_k + e_k$, we see $h_1 = \cdots = h_k = 0$. This proves that $e_0$ lies in the interior of $\tilde{\tau}$.

Thanks to Lemma~\ref{lemma:polarised_toric_proj} we have $\tilde{X} = \Proj \CC[ \tilde{\tau}^\vee \cap \tilde{M}_0]$ and $X = \Proj \CC[ \tau^\vee \cap M_0]$. The ring homomorphism
\begin{equation} \label{eq:ring_surjection_CC[tau_vee]}
\CC [\tilde{\tau}^\vee \cap \tilde{M}_0] \longrightarrow \CC[\tau^\vee \cap M_0],
\end{equation}
which is induced by the inclusion $\tau \into \tilde{\tau}$ and is surjective by the proof of Theorem~\ref{thm:mavlyutov}, is homogeneous with respect to the $\NN$-grading and induces a closed embedding $\iota \colon X \into \tilde{X}$. 
Using the isomorphisms \eqref{eq:anello_affine_sigma_isom_localizzazione _omogenea} it is not difficult to write down the formulae for the actions of the tori $T_N$ and $T_{\tilde{N}}$ on the affine charts of $X$ and $\tilde{X}$, respectively. From these formulae it is possible to see that $\iota$ is a toric morphism.

We have to prove that $X$ coincides with the closed subscheme of $\tilde{X}$ defined by the ideal $J_{\tilde{X}} \subseteq S_{\tilde{X}}$ generated by the binomials \eqref{eq:binomials_theorem_deformations_projective_varieties}. Let $J_{\tilde{C}} = J_{\tilde{X}} S_{\tilde{C}}$ be the extension of $J_{\tilde{X}}$ to the total coordinate ring $S_{\tilde{C}}$ of the affine cone $\tilde{C} = \Spec \CC[\tilde{\tau}^\vee \cap \tilde{M}_0]$ via the ring homomorphism $S_{\tilde{X}} \to S_{\tilde{C}}$ defined in Lemma~\ref{lemma:cox_coordinates_proj_toric_and_affine_cone}.
The ideal $J_{\tilde{C}}$ is generated generated by the binomials
\begin{equation} \label{eq:binomials_C_tilde_in_projective_varieties}
\prod_{\substack{\xi \in \tilde{\tau}(1) \colon \\ \langle e_i^*, \xi \rangle > 0}} x_\xi^{\langle e_i^*, \xi \rangle} - 
\prod_{\substack{\xi \in \tilde{\tau}(1) \colon \\ \langle e_i^*, \xi \rangle < 0}} x_\xi^{-\langle e_i^*, \xi \rangle} \qquad \text{for } i = 1, \dots, k
\end{equation}
in the Cox coordinates of $\tilde{C}$. By Theorem~\ref{thm:mavlyutov} the part of degree zero of $J_{\tilde{C}}$ in the ring $(S_{\tilde{C}})_0 \simeq \CC[\tilde{\tau}^\vee \cap \tilde{M}_0]$ coincides with the kernel $H$ of the ring surjection \eqref{eq:ring_surjection_CC[tau_vee]}. By Lemma~\ref{lemma:cox_coordinates_proj_toric_and_affine_cone}, $X = \Proj \CC[\tilde{\tau}^\vee \cap \tilde{M}_0]/ H$ coincides with the closed subscheme of $\tilde{X}$ defined by the ideal $J_{\tilde{X}}$.

The matrices $\cM_{\tilde{\Sigma}} = (\langle e_i^*, \rho \rangle)_{1 \leq i \leq k, \rho \in \tilde{\Sigma}(1)}$ and $\cM_{\tilde{\tau}} = (\langle e_i^*, \xi \rangle)_{1 \leq i \leq k, \xi \in \tilde{\tau}(1)}$ differ just by multiplication by a positive integer on each column, namely the numbers $b_\rho$ defined in Lemma~\ref{lemma:cox_coordinates_proj_toric_and_affine_cone}. From the proof of Lemma~\ref{lemma:cox_equations_slicing_cone} we see that $\cM_{\tilde{\tau}}$ has rank $k$ and each of its columns has at most one positive entry. Therefore also the matrix $\cM_{\tilde{\Sigma}}$ has these two properties. By Lemma~\ref{lemma:fischer_shapiro} or Remark~\ref{rmk:fischer_shapiro_new_proof}, the binomials \eqref{eq:binomials_theorem_deformations_projective_varieties} form a regular sequence.

This concludes the proof of Theorem~\ref{thm:deformations_projective_toric_varieties}(A).
\end{proof}

The following two lemmata should be well known, but we have not been able to find an adequate reference for them. 

\begin{lemma} \label{lemma:flatness_from_formal_to_local_proper}
Let $(A, \frakm)$ be a noetherian local ring and let $\pi \colon Y \to \Spec A$ be a proper morphism of schemes such that $Y \times_{\Spec A} \Spec A/\frakm^n \to \Spec A/\frakm^n$ is flat for every $n \in \NN$. Then $\pi$ is flat.
\end{lemma}

\begin{proof} This proof relies on an argument that appears in the proof of \cite[Proposition~6.51]{talpo_vistoli_deformation}.
We want to show that the set
$
Z = \{ y \in Y \mid \cO_{Y,y} \text{ is not flat over } A \}
$
is empty.
By covering $Y$ with open affine subschemes and by using \cite[Theorem 24.3]{matsumura}, one can see that $Z$ is closed in $Y$.

Assume by contradiction that $Z$ is non-empty. Since $\pi$ is closed, the set $\pi(Z)$ is a closed non-empty subset of $\Spec A$. Therefore $\frakm \in \pi(Z)$. Hence there exists $y_0 \in Z$ such that $\pi(y_0) = \frakm$. Let $\Spec R$ be an affine open neighbourhood of $y_0$ in $Y$ and let $B = \cO_{Y, y_0}$ be the local ring of $Y$ at $y_0$. We know that $A / \frakm^n \to R / \frakm^n R$ is flat for every $n \in \NN$. Therefore the local homomorphism $A \to B$ is such that $A / \frakm^n \to B / \frakm^n B$ is flat for every $n \in \NN$. By the local flatness criterion  \cite[Theorem~22.3]{matsumura} $A \to B$ is flat. But this is absurd because $y_0 \in Z$.
\end{proof}

\begin{lemma} \label{lemma:flatness_neighbourhood}
Let $S$ be a noetherian scheme and let $Y \to S$ be a scheme morphism of finite type such that $Y \times_S \Spec \cO_{S,s} \to \Spec \cO_{S,s}$ is flat for some point $s \in S$.
Then there exists an open neighbourhood $U$ of $s$ in $S$ such that $Y \times_S U \to U$ is flat.
\end{lemma}

\begin{proof}
Since the problem is local and $Y \to S$ is quasi-compact, we may assume $S = \Spec A$, $Y = \Spec B$ and $s = \frakm$ for some noetherian ring $A$, some finitely generated $A$-algebra $B$ and some prime ideal $\frakm$ of $A$. We know that $B \otimes_A A_\frakm$ is flat over $A_\frakm$.
Let us consider the set
\begin{align*}
V = \{ P \in \Spec B \mid B_P \text{ is flat over } A_{P \cap A} \}
= \{ P \in \Spec B \mid B_P \text{ is flat over } A \},
\end{align*}
which is open in $\Spec B$ by \cite[Theorem 24.3]{matsumura}. The equality above holds by transitivity of flatness and \cite[Theorem 7.1]{matsumura}.

We identify $\Spec (B \otimes_A A_\frakm)$ with the set of primes $P \in \Spec B$ such that $P \cap A \subseteq \frakm$.
If $P \in \Spec B$ is such that $P \cap A \subseteq \frakm$, then by \cite[Theorem 7.1]{matsumura} from the flatness of $B \otimes_A A_\frakm$ over $A_\frakm$ we deduce that $B_P$ is flat over $(A_\frakm)_{(P \cap A)A_\frakm} = A_{P \cap A}$. This shows that $\Spec (B \otimes_A A_\frakm)$ is contained in $V$.

Consider the set $A \setminus \frakm$ endowed with the order relation $\leq$ such that $f \leq g$ if and only if $g \in \sqrt{Af}$. If $f \leq g$, there is the localisation map $A_f \to A_g$, given by the restriction of the structure sheaf of $\Spec A$ from the principal open subset defined by $f$ to the principal open subset defined by $g$.
As $f$ runs in $A \setminus \frakm$, the rings $A_f$  form a direct system and the local ring $A_\frakm$ is the direct limit of this system. Since tensor products and direct limits commute, $B \otimes_A A_\frakm$ is the limit of $B_f$ as $f \in A \setminus \frakm$. We are in the situation of inverse limits of affine schemes studied in \cite[\S{8}]{ega4_3}, i.e.\ $\Spec (B \otimes_A A_\frakm)$ is the projective limit of the affine schemes $\Spec B_f$ as $f$ runs in $A \setminus \frakm$.

For every $f \in A \setminus \frakm$, consider the set $E_f = V \cap \Spec B_f$, which is open in $\Spec B_f$ because $V$ is open in $\Spec B$. Since $\Spec (B \otimes_A A_\frakm)$ is contained in $V$, the set $E = V \cap \Spec (B \otimes_A A_\frakm)$ coincides with $\Spec (B \otimes_A A_\frakm)$. Since $E$ is the limit of the $E_f$'s, by \cite[Corollaire 8.3.5]{ega4_3} we have that there exists $f_0 \in A \setminus \frakm$ such that $E_{f_0} = \Spec B_{f_0}$. This implies that $B_{f_0}$ is flat over $A_{f_0}$. Therefore we may take $U = \Spec A_{f_0}$.
\end{proof}

\begin{proof}[Proof of Theorem \ref{thm:deformations_projective_toric_varieties}(B)]
The proof of the fact that the trinomials~\eqref{eq:trinomials_theorem_deformations_projective_varieties} are elements of $\CC[t_1, \dots, t_k][x_\rho \mid \rho \in \tilde{\Sigma}(1)]$ is completely analogous to what is done in the proof of Theorem~\ref{thm:mavlyutov}(B) and will be omitted.

Let $\cX$ be the closed subscheme of $\tilde{X} \times_{\Spec \CC} \AA^k_\CC$ defined by the homogeneous ideal generated by the trinomials~\eqref{eq:trinomials_theorem_deformations_projective_varieties}. 
By (A) the fibre of $\cX \to \AA^k_\CC$ over the origin is $X$.
The fibred product $\cX \times_{\AA^k_\CC} \Spec \CC[t_1, \dots, t_k] / \frakq$ is flat over $\CC[t_1, \dots, t_k] / \frakq$ for every $(t_1, \dots, t_k)$-primary ideal $\frakq$ of $\CC[t_1, \dots, t_k]$, thanks to Lemma~\ref{lemma:complete_intersection_flatness_neighbourhood} and Lemma~\ref{lemma:flatness_sheafification_graded_module}, as in the proof of Theorem~\ref{thm:mavlyutov}(B). If $A = \CC[t_1, \dots, t_k]_{(t_1, \dots, t_k)}$ is the local ring of $\AA^k_\CC$ at the origin $O$, by Lemma~\ref{lemma:flatness_from_formal_to_local_proper} the morphism $\cX \times_{\AA^k_\CC} \Spec A \to \Spec A$ is flat, and consequently it induces a deformation of $X$ over $\hat{A} = \CC [ \! [ t_1, \dots, t_k ] \! ]$. By Lemma~\ref{lemma:flatness_neighbourhood} we may find an open neighbourhood $U \subseteq \AA^k_\CC$ of $O$ such that $\cX \times_{\AA^k_\CC} U$ is flat over $U$.
\end{proof}

\section{Deformations of projective toric pairs}
\label{sec:deformations_projective_toric_pairs}

In Theorem~\ref{thm:deformations_projective_toric_varieties}, from a projective toric variety $X$ with an ample $\QQ$-Cartier $\QQ$-divisor $D$ and a Minkowski decomposition of a certain polyhedron with some properties we constructed a deformation of $X$. Here we show that if $D$ is a $\ZZ$-divisor then we can construct a deformation of the toric pair $(X,\partial X)$. This is the content of the following theorem.

\begin{theorem}\label{thm:deformations_projective_toric_pairs}
Let $N$ be a lattice of rank $n$, let $X$ be a projective $T_N$-toric variety, and let $D$ be an ample torus-invariant $\QQ$-Cartier $\ZZ$-divisor on $X$. Let $\tau$ be the $(n+1)$-dimensional cone in the lattice $N_0 = N \oplus \ZZ e_0$ associated to the polarised projective toric variety $(X,D)$ as in Lemma~\ref{lemma:polarised_projective_varieties}. Let $(Q, Q_0, Q_1, \dots, Q_k, w)$ be a $\partial$-deformation datum for $(N_0, \tau)$ with $w \in M \subseteq M_0$.
Consider the lattices $\tilde{N} = N \oplus \ZZ e_1 \oplus \cdots \oplus \ZZ e_k$ and $\tilde{N}_0 = N \oplus \ZZ e_0 \oplus \ZZ e_1 \oplus \cdots \oplus \ZZ e_k$. Let $\tilde{\tau} \subseteq (\tilde{N}_0)_\RR$ and $\tilde{w} \in \tilde{M} \subseteq \tilde{M}_0$ be as in Notation~\ref{not:tilde_N_tilde_sigma_tilde_w}. 

Let $\partial X$ be the toric boundary of $X$. Let $(\tilde{X}, \tilde{D})$ be the polarised projective toric variety associated to the cone $\tilde{\tau}$ via Lemma \ref{lemma:polarised_projective_varieties}.
Consider the reduced effective divisor $\cD$ on $\tilde{X}$ defined by the homogeneous ideal generated by the following monomial in the Cox coordinates of $\tilde{X}$:
\begin{equation} \label{eq:monomial_theorem_deformations_projective_pairs}
\prod_{\substack{\rho \in \tilde{\Sigma}(1) \colon \\ \forall i \in \{ 1, \dots, k \}, \langle e_i^*, \rho \rangle \leq 0}} x_\rho,
\end{equation}
where $\tilde{\Sigma}$ is the fan of $\tilde{X}$ in $\tilde{N}$.

Let $t_1, \dots, t_k$ be the standard coordinates on $\AA^k_\CC$. Consider the closed subscheme $\cX$ of $\tilde{X} \times_{\Spec \CC} \AA^k_\CC = \toric{\CC[t_1, \dots, t_k]}{\tilde{\Sigma}}$ defined by the homogeneous ideal generated by the following trinomials in Cox coordinates:
\begin{equation} \label{eq:trinomials_theorem_deformations_projective_pairs}
\prod_{\substack{\rho \in \tilde{\Sigma}(1) \colon \\ \langle e_i^*, \rho \rangle > 0}} x_\rho^{\langle e_i^*, \rho \rangle} - 
\prod_{\substack{\rho \in \tilde{\Sigma}(1) \colon \\ \langle e_i^*, \rho \rangle < 0}} x_\rho^{-\langle e_i^*, \rho \rangle} -
t_i \prod_{\rho \in \tilde{\Sigma}(1)} x_\rho^{\langle \tilde{w}, \rho \rangle} \prod_{\substack{\rho \in \tilde{\Sigma}(1) \colon \\ \langle e_i^*, \rho \rangle < 0}} x_\rho^{-\langle e_i^*, \rho \rangle}
\end{equation}
for $i=1, \dots, k$. 

Then the diagram
\[
\xymatrix{
\cX \cap (\cD \times_{\Spec \CC} \AA^k_\CC) \ar[rd] \ar@{^{(}->}[r] & \cX \ar[d] \\
& \AA^k_\CC
}
\]
induces a deformation of the pair $(X, \partial X)$ over $\CC [ \![ t_1, \dots, t_k ] \! ]$ and over some open neighbourhood of the origin in $\AA^k_\CC$.
\end{theorem}

We mean that the base change of the diagram above to the origin of $\AA^k_\CC$ is the closed embedding $\partial X \into X$ over $\Spec \CC$ and that we get flat families when we base change the morphisms $\cX \cap (\cD \times_{\Spec \CC} \AA^k_\CC) \to \AA^k_\CC$ and $\cX \to \AA^k_\CC$ to $\Spec \CC [ \! [ t_1, \dots, t_k ] \! ]$ and to some open neighbourhood of the origin in $\AA^k_\CC$.

\begin{proof}[Proof of Theorem~\ref{thm:deformations_projective_toric_pairs}]
By Theorem~\ref{thm:deformations_projective_toric_varieties} it is enough to deal with the toric boundary.
Let $C$ and $\tilde{C}$ be the affine cones over $X$ and $\tilde{X}$ as in the proof of Theorem~\ref{thm:deformations_projective_toric_varieties}.
Let $\overline{J}_{\tilde{X}}$ be the ideal in the Cox ring of $\tilde{X}$ generated by the binomials \eqref{eq:binomials_theorem_deformations_projective_varieties} and the monomial \eqref{eq:monomial_theorem_deformations_projective_pairs}. We need to show that the closed subscheme of $\tilde{X}$ defined by $\overline{J}_{\tilde{X}}$ coincides with $\partial X$.

Since $X$ is polarised by a $\ZZ$-divisor, by Lemma~\ref{lemma:polarised_projective_varieties} the primitive generator of every ray of the cone $\tau$ is of the form $\rho - a_\rho e_0$ for some $a_\rho \in \ZZ$ and $\rho \in N$ primitive. From the definition of $\tilde{\tau}$ in Notation~\ref{not:tilde_N_tilde_sigma_tilde_w} it is easy to see that also $\tilde{\tau}$ has the same property, i.e.\ the primitive generator of every ray of $\tilde{\tau}$ is of the form $\rho - a_\rho e_0$ for some $a_\rho \in \ZZ$ and $\rho \in \tilde{N}$ primitive. Therefore the homomorphism $S_{\tilde{X}} \to S_{\tilde{C}}$, defined in Lemma~\ref{lemma:cox_coordinates_proj_toric_and_affine_cone}, is the identity. In particular the extended ideal $\overline{J}_{\tilde{C}} := \overline{J}_{\tilde{X}} S_{\tilde{C}} \subseteq S_{\tilde{C}}$ is generated by the binomials \eqref{eq:binomials_C_tilde_in_projective_varieties} and the monomial
\begin{equation*} 
\prod_{\substack{\xi \in \tilde{\tau}(1) \colon \\ \forall i \in \{ 1, \dots, k \}, \langle e_i^*, \xi \rangle \leq 0}} x_\xi.
\end{equation*}
By Theorem~\ref{thm:deformations_affine_toric_pairs} the contraction of $\overline{J}_{\tilde{C}}$ to the degree zero part $(S_{\tilde{C}})_0 \simeq \CC[\tilde{\tau}^\vee \cap \tilde{M}_0]$ of $S_{\tilde{C}}$ coincides with the ideal $\tilde{L}$ of $\partial C$ in $\tilde{C}$, which is the preimage along the surjection \eqref{eq:ring_surjection_CC[tau_vee]} of the the ideal $L$ of $\partial C$ in $C$. By Lemma~\ref{lemma:cox_coordinates_proj_toric_and_affine_cone} the closed subscheme of $\tilde{X}$ defined by $\overline{J}_{\tilde{X}}$ coincides with
\[
\Proj \CC[\tilde{\tau}^\vee \cap \tilde{M}_0]/\tilde{L} = \Proj \CC[\tau^\vee \cap M_0]/L = \partial X,
\]
where the last equality follows from Lemma~\ref{lemma:polarised_toric_proj}.

\begin{equation*}
\xymatrix{
L \ar@{^{(}->}[d] &
 \tilde{L} \ar@{^{(}->}[d] \ar@{->>}[l] \ar@{=}[r] &
  \overline{J}_{\tilde{C}} \cap (S_{\tilde{C}})_0 \ar@{^{(}->}[d] \ar@{^{(}->}[r] &
   \overline{J}_{\tilde{C}} \ar@{^{(}->}[d] &
    \overline{J}_{\tilde{X}} \ar@{^{(}->}[d] \ar[l]_{\simeq} \\
\CC[\tau^\vee \cap M_0] &
 \CC[\tilde{\tau}^\vee \cap \tilde{M}_0] \ar@{->>}[l]_{\ \eqref{eq:ring_surjection_CC[tau_vee]}} \ar@{=}[r]^{\quad\ \mathrm{Cox}_{\tilde{\tau}}} &
  (S_{\tilde{C}})_0 \ar@{^{(}->}[r] &
   S_{\tilde{C}} &
    S_{\tilde{X}} \ar[l]_{\simeq}
}
\end{equation*}

One can show that the binomials \eqref{eq:binomials_theorem_deformations_projective_varieties} and the monomial \eqref{eq:monomial_theorem_deformations_projective_pairs} form a regular sequence by using Lemma~\ref{lemma:regular_sequence_binomials_monomial}. By adapting the proof Theorem~\ref{thm:deformations_projective_toric_varieties}(B) we can show the flatness of the families.
\end{proof}

\section{Mutations of Fano polytopes and deformations of Fano toric pairs} \label{sec:mutations_induce_deformations}

Here we recall the definition of mutations between Fano polytopes from \cite{sigma} and we prove Theorem~\ref{thm:mutations_induce_deformations}. The definition of Fano polytope has been given at the beginning of \S\ref{sec:intro_mutations_deformations}.

If $N$ is a lattice, $w \in M_\RR \setminus \{ 0 \}$ and $h \in \RR$, then we denote by $H_{w,h}$ the set of all points of $N_\RR$ lying at height $h$ with respect to $w$, i.e.\ the affine hyperplane $H_{w,h} := \{ v \in N_\RR \mid \langle w, v \rangle = h \}$.  In particular $w^\perp = H_{w,0}$.

\begin{definition} \label{def:factor}
Let $P \subseteq N_\RR$ be a Fano polytope. A \emph{mutation datum for $P$} is a pair $(w,F)$ where $w \in M$ is a primitive vector and $F \subseteq w^\perp \subseteq N_\RR$ is a lattice polytope satisfying the following condition:  for every $h \in \ZZ$ such that $\min_P  w \leq h < 0$, there exists a (possibly empty) lattice polytope $G_h \subseteq N_\RR$ such that
\begin{equation} \label{eq:conditions_factor}
H_{w,h} \cap \vertices(P) \subseteq G_h + (-h) F \subseteq \conv{ H_{w,h} \cap P \cap N}.
\end{equation}  
\end{definition}

Note that, for given Fano polytope $P \subseteq N_\RR$ and primitive vector $w \in M$, a polytope $F$ such that $(w,F)$ is a mutation datum for $P$ need not exist. From a mutation datum we make the following construction.

\begin{definition}[{\cite[Definition 5]{sigma}}] \label{def:mutation}
Let $P \subseteq N_\RR$ be a Fano polytope and let $(w,F)$ be a mutation datum for $P$. Assume that $\{ G_h \}_{\min_P w \leq h < 0}$ is a collection of lattice polytopes satisfying \eqref{eq:conditions_factor}. We define the corresponding \emph{mutation} to be the lattice polytope
\begin{equation*}
\mathrm{mut}_{w,F} (P) := \conv{ \bigcup_{h =  \min_P w}^{-1} G_h \cup \bigcup_{h=0}^{h_\mathrm{max}}  ((H_{w,h} \cap P \cap N) +hF) }.
\end{equation*}
\end{definition}

The polytope $\mathrm{mut}_{w,F}(P)$ does not depend on the choice of $\{ G_h \}$. Moreover, $\mathrm{mut}_{w,F}(P)$ is a Fano polytope. See \cite[\S{3}]{sigma} or \cite[\S{2.5}]{mohammad_thesis} for the proofs of these statements.

Roughly speaking, $\mut_{w,F}(P)$ is obtained from $P$ by adding $hF$ at height $h$ (with respect to $w$) for $h > 0$ and by removing $(-h) F$ at height $h$ for $h <0$.
The pair $(w,F)$ is a mutation datum precisely when it is possible to remove from $P$ multiples of $F$ at negative heights.
For an example of mutation of Fano polytopes see the beginning of Example~\ref{ex:P2_P114}.

The following theorem is the precise version of Theorem~\ref{thm:mutations_induce_deformations_easy}.

\begin{theorem} \label{thm:mutations_induce_deformations}
Let $P \subseteq N_\RR$ be a Fano polytope, let $(w, F)$ be a mutation datum for $P$, and let $P' = \mathrm{mut}_{w,F}(P)$ be the mutated polytope.
Let $X_P$ (resp.\ $X_{P'}$) be the Fano toric variety associated to the spanning fan of $P$ (resp.\ $P'$) and let $\partial X_P$ (resp.\ $\partial X_{P'}$) be the toric boundary of $X_P$ (resp.\ $X_{P'}$).
Set
\begin{align*}
\vertices(P)^{\geq 0} &= \vertices(P) \cap \{ v \in N \mid \langle w, v \rangle \geq 0 \}, \\
\vertices(P')^{< 0} &= \vertices(P') \cap \{ v \in N \mid \langle w, v \rangle < 0 \}.
\end{align*}
Consider the lattice $\tilde{N} = N \oplus \ZZ e_1$ and the polyhedron $\tilde{Q} \subseteq \tilde{M}_\RR$ defined by
\begin{equation*}
\tilde{Q} = \left\{ u + k e_1^* \in \tilde{M}_\RR \left\vert  \begin{matrix}
\ \forall p \in \vertices(P)^{\geq 0}, \quad \langle u, p \rangle + 1 \geq 0 \\
\ \forall p' \in \vertices(P')^{< 0}, \quad \langle u, p' \rangle + 1 + k \langle w, p' \rangle \geq 0 \\
\ \forall f \in \vertices(F), \quad \langle u, f \rangle + k \geq 0
\end{matrix}   \right.   \right\}.
\end{equation*}
Then $\tilde{Q}$ is a full dimensional rational polytope and the primitive generators of the rays of the normal fan $\tilde{\Sigma}$ of $\tilde{Q}$ are
\begin{itemize}
\item $p$ for $p \in \vertices(P)^{\geq 0}$,
\item $p' + \langle w , p' \rangle e_1$ for $p' \in \vertices(P')^{<0}$,
\item $f + e_1$ for $f \in \vertices(F)$.
\end{itemize}
Let $\tilde{X} = \toric{\CC}{\tilde{\Sigma}}$ be the toric variety associated to $\tilde{\Sigma}$. Consider the reduced divisor $\cD$ on $\tilde{X}$ defined by the following monomial in the Cox coordinates of $\tilde{X}$:
\begin{equation} \label{eq:monomial_theorem_mutation}
\prod_{p \in \vertices(P)^{\geq 0}} x_p
\prod_{p' \in \vertices(P')^{< 0}} x_{p'}.
\end{equation}
Set  $V = \PP^2_\CC \setminus \{ [1:0:0],[0:1:0]\}$. Consider the closed subscheme $\cX$ of $\tilde{X} \times_{\Spec \CC} V$ defined by the vanishing of the trinomial obtained by varying the three coefficients of
\begin{equation} \label{eq:trinomial_mutations}
\prod_{p \in \vertices(P)^{\geq 0}} x_p^{\langle w, p \rangle} +
\prod_{p' \in \vertices(P')^{< 0}} x_{p'}^{-\langle w, p' \rangle} +
\prod_{f \in \vertices(F)} x_f.
\end{equation}
Then in the diagram
\[
\xymatrix{
\cX \cap (\cD \times_{\Spec \CC} V) \ar[rd] \ar@{^{(}->}[r] & \cX \ar[d] \\
& V
}
\]
the two morphisms with target $V$ are flat and the base change of the diagram to the points $[0 : 1 :-1]$ and $[1:0:-1]$ of $V$ are the closed embeddings $\partial X_P \into X_P$ and $\partial X_{P'} \into X_{P'}$ over $\Spec \CC$, respectively.
\end{theorem}

When $\dim X_P = 2$, a slight variation of this construction was pursued in \cite[Lemma~7]{conjectures}.  The rays of the fan of $\tilde{X}$ in Theorem~\ref{thm:mutations_induce_deformations} have been suggested to us by Thomas Prince. An example of Theorem~\ref{thm:mutations_induce_deformations} is given below.

\begin{example} \label{ex:P2_P114}
	In the lattice $N = \ZZ^2$
	\begin{figure}
		\centering
		\includegraphics[width=7cm]{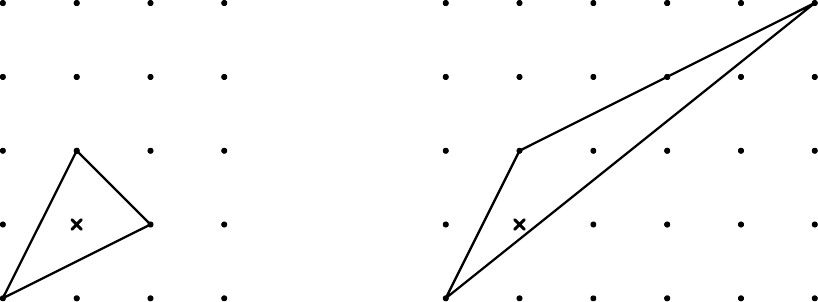}
		\caption{The polygons in Example~\ref{ex:P2_P114} \label{fig:polygons}}
	\end{figure}
	consider the polygons
	\begin{align*}
	P &= \conv{
        \begin{pmatrix}
        1 \\ 0
        \end{pmatrix},
        \begin{pmatrix}
        0 \\ 1
        \end{pmatrix},
        \begin{pmatrix}
        -1 \\ -1
        \end{pmatrix}
}  \\
	P' &= \conv{
                \begin{pmatrix}
        4 \\ 3
        \end{pmatrix},
        \begin{pmatrix}
        0 \\ 1
        \end{pmatrix},
        \begin{pmatrix}
        -1 \\ -1
        \end{pmatrix}
    }
	\end{align*}
	which appear in Figure~\ref{fig:polygons}.
	We have $P' = \mathrm{mut}_{w,F}(P)$ with $w = (-1,2) \in M$ and 
    \[
    F = \conv{
\begin{pmatrix}
0 \\ 0
\end{pmatrix}, 
\begin{pmatrix}
2 \\ 1
\end{pmatrix}
}.
\]
	We see that $X_P = \PP^2$ and $X_{P'} = \PP(1,1,4)$. Following Theorem~\ref{thm:mutations_induce_deformations} we consider the complete fan $\tilde{\Sigma}$ in $\tilde{N} = \ZZ^3$ whose rays are generated by
	\begin{equation*}
	x = \begin{pmatrix}
	0 \\ 1 \\ 0
	\end{pmatrix},
	y =\begin{pmatrix}
	-1 \\ -1 \\ -1
	\end{pmatrix},
	z_0 =\begin{pmatrix}
	0 \\ 0 \\ 1
	\end{pmatrix},
	z_1 = \begin{pmatrix}
	2 \\ 1 \\ 1
	\end{pmatrix}.
	\end{equation*}
	This implies that the toric variety associated to $\tilde{\Sigma}$ is $\tilde{X} = \PP(1,2,1,1)$ with Cox coordinates $x,y,z_0, z_1$. The trinomial \eqref{eq:trinomial_mutations} is $x^2 + y + z_0 z_1$ and the monomial \eqref{eq:monomial_theorem_mutation} is $xy$.
	
Consider $V = \PP^2_\CC \setminus \{[1:0:0],[0:1:0]\}$ with homogeneous coordinates $a,b,c$.
Consider the following closed subschemes of $\PP(1,2,1,1) \times_{\Spec \CC} V$:
	\begin{align*}
	\cB = \{ a x^2 + by + c z_0 z_1 = xy = 0 \} \into \cX = \{ a x^2 + by + c z_0 z_1 = 0 \}.
	\end{align*}
	By Theorem~\ref{thm:mutations_induce_deformations}, the projections $\cX \to V$ and $\cB \to V$ are flat and their fibres over $[0:1:-1]$ (resp.\ $[1:0:-1]$) are $X_P$ and $\partial X_P$ (resp.\ $X_{P'}$ and $\partial X_{P'}$). We notice that the fibres over the closed points of $V$ with $b \neq 0$ are $\PP^2_\CC$ with a reducible cubic.
\end{example}

\begin{lemma} \label{lemma:trinomial_monomial_regular_sequence}
Consider the polynomial ring $S = \CC[x_1, \dots, x_r, y_1, \dots, y_s, z_1, \dots, z_t]$ over $\CC$. Fix $\alpha_1, \dots, \alpha_r, \beta_1, \dots, \beta_s \in \NN$.
 If $a,b,c \in \CC$ are such that $ab \neq 0$ or $c \neq 0$, then
\[
a x_1^{\alpha_1} \cdots x_r^{\alpha_r} + b y_1^{\beta_1} \cdots y_s^{\beta_s} + c z_1 \cdots z_t, \ x_1 \cdots x_r y_1 \cdots y_s
\]
 is a regular sequence in $S$.
\end{lemma}

\begin{proof} They are two coprime elements in a unique factorisation domain.
\end{proof}

\begin{proof}[Proof of Theorem~\ref{thm:mutations_induce_deformations}]
Consider the lattice $N_0 = N \oplus \ZZ e_0$ and the cone $\tau = \cone{P+e_0} \subseteq (N_0)_\RR$, which is associated to the ample $\QQ$-Cartier $\ZZ$-divisor $\partial X_P$ on $X_P$ via Lemma~\ref{lemma:polarised_projective_varieties}. Let $\{ G_h \}$ be the collection of lattice polytopes in Definition~\ref{def:factor}. Consider the polytope
\begin{equation*}
G = \conv{ \bigcup_{\min_P w \leq h < 0} \frac{G_h + e_0}{-h} }.
\end{equation*}
It is not difficult to see that $(G+F, G,F,w)$ is a $\partial$-deformation datum for $(N_0, \tau)$ with $w \in M \subseteq M_0$.

Let us prove only (vi). Each vertex of the polyhedron $\tau \cap \{ v + k e_0 \in (N_0)_\RR \mid \langle w, v \rangle = -1 \}$ is of the form $- \langle w, p \rangle\inv (p+e_0)$ for some $p \in \vertices(P)^{<0}$; by \eqref{eq:conditions_factor} there exist $g \in G_{\langle w, p \rangle}$ and $f \in F$ such that $p = g - \langle w, p \rangle f$; this implies that
\begin{equation*}
- \frac{p+e_0}{\langle w, p \rangle} = \frac{g + e_0}{- \langle w, p \rangle} + f \in G + F.
\end{equation*}

Now consider the cone $\tilde{\tau} = \cone{\tau, G-e_1, F+e_1}$ in $\tilde{N}_0 = N \oplus \ZZ e_0 \oplus \ZZ e_1$, as in Notation~\ref{not:tilde_N_tilde_sigma_tilde_w}. By using \eqref{eq:conditions_factor} it is not difficult to show that $\tilde{\tau}$ is generated by $p + e_0$ for $p \in \vertices(P)^{\geq 0}$,
$p' + e_0 + \langle w , p' \rangle e_1$ for $p' \in \vertices(P)^{<0}$, and 
$f + e_1$ for $f \in \vertices(F)$. It is not difficult to show that these are the rays of $\tilde{\tau}$. This implies that the polytope $\tilde{Q}$ defined in Theorem~\ref{thm:mutations_induce_deformations} coincides with $\tilde{\tau}^\vee \cap e_0\inv(1)$ and that the rays of $\tilde{\Sigma}$ are the ones written down in the statement of the theorem.

Consider $a,b,c$ homogeneous coordinates on $V \subseteq \PP^2_\CC$, consider the closed subscheme $\cX$ of $\tilde{X} \times_{\Spec \CC} V$ defined by the homogeneous ideal generated by the trinomial
\begin{equation} \label{eq:trinomial_abc_mutation}
a \prod_{p \in \vertices(P)^{\geq 0}} x_p^{\langle w, p \rangle} +
b \prod_{p' \in \vertices(P')^{< 0}} x_{p'}^{-\langle w, p' \rangle} +
c \prod_{f \in \vertices(F)} x_f.
\end{equation}
Let $\cD$ be reduced effective divisor of $\tilde{X}$ defined by the monomial \eqref{eq:monomial_theorem_mutation}. Consider $\cB = \cX \cap (\cD \times_{\Spec \CC} V) \into \tilde{X} \times_{\Spec \CC} V$. Therefore the diagram in Theorem~\ref{thm:mutations_induce_deformations} is $\cB \into \cX \to V$.

If we ignore the coefficients, the trinomial \eqref{eq:trinomials_theorem_deformations_projective_pairs} in this particular case becomes  the trinomial in Theorem~\ref{thm:mutations_induce_deformations} and the monomial \eqref{eq:monomial_theorem_deformations_projective_pairs} becomes the monomial \eqref{eq:monomial_theorem_mutation}. By Theorem~\ref{thm:deformations_projective_toric_pairs}, the fibres of $\cX \to V$ and $\cB \to V$ over the point $[0:1:-1]$ are $X_P$ and $\partial X_P$.

Now we prove that $\cX$ and $\cB$ are flat over $V$. We pick an arbitrary closed point $v_0 = [a:b:c]$ of $V$. By Lemma~\ref{lemma:trinomial_monomial_regular_sequence}, the trinomial~\eqref{eq:trinomial_abc_mutation} and the monomial~\eqref{eq:monomial_theorem_mutation} form a regular sequence
 in the Cox ring of $\tilde{X}$. By Lemma~\ref{lemma:complete_intersection_flatness_neighbourhood} and Lemma~\ref{lemma:flatness_sheafification_graded_module}, we have that the morphisms $\cX \to V$ and $\cB \to V$ become flat when we restrict them to every infinitesimal neighbourhood of the point $v_0$ in $V$. Since $\cX$ and $\cB$ are proper over $V$, by 
 Lemma~\ref{lemma:flatness_from_formal_to_local_proper} we have that the morphisms $\cX \to V$ and $\cB \to V$ become flat when we base change them to $\Spec \cO_{V,v_0}$. By Lemma~\ref{lemma:flatness_neighbourhood} there exists an open neighbourhood $U$ of $v_0$ in $V$ such that $\cX \vert_U$ and $\cB \vert_U$ are flat over $U$. Therefore $\cX$ and $\cB$ are flat over $V$.

It remains to show that $X_{P'}$ and $\partial X_{P'}$ are the fibres of $\cX \to V$ and $\cB \to V$ over $[1:0:-1]$. This can be seen by using the inverse mutation from $P'$ to $P$ and by applying the automorphism of $N \oplus \ZZ e_1$ given by $v + ke_1 \mapsto v + (k + \langle w, v \rangle) e_1$.
\end{proof}

\bibliography{Biblio}


\end{document}